\documentclass[11pt]{amsart}
\usepackage{verbatim}
\usepackage{latexsym}
\usepackage{amsmath,amsthm,amssymb, amscd,color,url,mathtools}
\usepackage{mathrsfs}
\usepackage[all]{xy}
\usepackage[all]{xy}
\usepackage{tikz}
\usepackage{tikz-cd}
\usepackage{upgreek}
\usetikzlibrary{decorations.pathmorphing}
\usepackage{anysize}
\usepackage{enumerate}
\marginsize{2.5cm}{2.5cm}{2.5cm}{2.5cm}

\usepackage[colorlinks=true,
            citecolor={red!60!black},
            linkcolor=blue,
            urlcolor=blue]{hyperref}

\input xy \xyoption{all}
\usepackage{blkarray}  
\usepackage{framed}
\usepackage[utf8]{inputenc}

\usepackage{soul}

\theoremstyle{definition}
\newtheorem{theorem}{Theorem}[section]
\newtheorem{proposition}[theorem]{Proposition}
\newtheorem{lemma}[theorem]{Lemma}
\newtheorem{corollary}[theorem]{Corollary}
\newtheorem{definition}[theorem]{Definition}

\newtheorem{remark}[theorem]{Remark}

\newtheorem{example}[theorem]{Example}

\newtheorem{subsec}[theorem]{}

\newtheorem*{thma}{Theorem A}
\newtheorem*{thmb}{Theorem B}
\newtheorem*{thmc}{Theorem C}

\newenvironment{myeq}[1][]
{\stepcounter{theorem}\begin{equation}\tag{\thetheorem}{#1}}
{\end{equation}}

\newenvironment{mysubsection}[2][]
{\begin{subsec}\begin{upshape}\begin{bfseries}{#2.}
			\end{bfseries}{#1}}
		{\end{upshape}\end{subsec}}

\newcommand{\RR}{\mathbb{R}}

\newcommand{\NN}{\mathbb{N}}

\newcommand{\cL}{\mathcal{L}}
\newcommand{\Rn}{\mathbb{R}^n}

\newcommand{\K}{\mathcal{K}}
\newcommand{\Top}{{\textbf {Top} }}
\newcommand{\GTop}{{\textbf {$G$-Top} }}
\newcommand{\GTR}{{\textbf {$G$-Top$^\mathbb{R}$}}}
\newcommand{\GTI}{{\textbf {$G$-Top$^\textbf{I}$}}}
\newcommand{\TI}{{\textbf {Top$^\textbf{I}$}}}

\newcommand{\I}{{\textbf {{I}}}}
\newcommand{\GTIC}{\GTI_{crit}}
\newcommand{\Q}{{\mathbf{Q}}}
\newcommand{\sSet}{{\textbf{sSet}}}
\newcommand{\GsS}{{\textbf{$G$-sSet}}}
\newcommand{\GSR}{{\text{$G$-$\mathbb{S}^{\RR}$}}}
\newcommand{\colim}{{\mbox{colim}}}

\newcommand{\fcolim}{{\mbox{\underline{colim}}_G}}
\newcommand{\Nrv}{{\textbf{Nrv}_G}}
\newcommand{\nrv}{{\textbf{Nrv}}}

\newcommand{\Orb}{{\mbox{Orb}}}
\newcommand{\Ho}{\operatorname{Ho}}
\newcommand{\GH}{{\text{GH}}}
\newcommand{\HI}{{\text{HI}}}
\newcommand{\HC}{{\text{HC}}}
\newcommand{\IHC}{{\text{IHC}}}

\newcommand{\C}{\operatorname{\mathcal{C}}}
\newcommand{\VR}{\operatorname{VR}}
\newcommand{\sd}{\operatorname{sd}}
\newcommand{\U}{\mathcal{U}}
\newcommand{\br}{\operatorname{Bar}}

\makeatletter \def\blfootnote{\xdef\@thefnmark{}\@footnotetext}\makeatother

\numberwithin{equation}{section}

\numberwithin{equation}{section}

\begin{document}

\title[$G$-homotopy interleaving and $G$-persistent Whitehead theorem]{Equivariant homotopy interleaving\\ and persistent Whitehead theorem}

\author[K. Halder]{Kushal Halder}
\address{Department of Mathematics and Statistics, Indian Institute of Technology Kanpur, India}
\email{kushalh24@iitk.ac.in / halderkushalr24@gmail.com}

\author[S. Sau]{Subhankar Sau}
\address{Department of Mathematics and Statistics, Indian Institute of Technology Kanpur, India}
\email{subhankarsau18@gmail.com}

\author[D.Sen]{Debasis Sen}
\address{Department of Mathematics and Statistics, Indian Institute of Technology Kanpur, India}
\email{debasis@iitk.ac.in}



\date{\today}
\dedicatory{}

\abstract 
    Let $G$ be a finite group. In this article, we develop some results of equivariant persistent homotopy theory for persistent $G$-spaces. We introduce the notions of $G$-stable and $G$-homotopy invariant distances and define the $G$-homotopy interleaving distance, an equivariant analogue $d^G_{HI}$ of the homotopy interleaving distance of Blumberg and Lesnick. We prove that this distance $d^G_{HI}$ is $G$-stable and $G$-homotopy invariant and establish its universality by showing that it dominates any such distance. We further prove a version of equivariant persistent Whitehead theorem relating interleavings of equivariant persistent homotopy groups to $G$-homotopy interleavings of persistent $G$-spaces. We also prove an equivariant persistent nerve lemma for equivariant good covers of persistent $G$-spaces. As a consequence, we obtain an equivariant weak law of large numbers for filtrations, providing a homotopy-theoretic consistency result for random filtrations equipped with finite group of symmetries. 
\endabstract

\maketitle


\section{Introduction}

Persistent homology is an important tool in topological data analysis (TDA), providing computable invariants that capture the topological features of data, via filtered spaces. Edelsbrunner--Letscher--Zomorodian introduced topological persistence in  \cite{ELZ02}. Since then persistent homology has been widely used in areas such as geometry, data analysis, and machine learning. The stability theorem of Cohen-Steiner--Edelsbrunner--Harer \cite{CEH07}, together with the subsequent development of the notion of interleaving (see \cite{CSGO16, Lesnick15, BL23, LS23}), has provided a systematic way to compare filtrations and persistence modules. The interleaving distance allows us to measure how close two such filtrations are. This gives a way to study how topological features appear and persist at different scales.

In TDA, filtrations can be viewed as diagrams indexed by the real line or, more generally, by a partially ordered set. Applying the homology functor to a filtered diagram of topological spaces gives an algebraic way to study the topological features of the spaces arising from data, such as connected components, holes, and voids. While persistent homology captures important geometric information, it does not retain the complete homotopy type of a filtration. This limitation has motivated the development of persistent homotopy theory, where filtrations are studied directly in suitable homotopy categories. A major step in this direction was taken by Jardine\cite{Jar20} and Blumberg--Lesnick\cite{BL23} in their work.

Many spaces arising in geometry, topology, and applications naturally possess symmetries arising from rotations, reflections, permutations, or actions of groups. Examples include molecular structures, crystalline materials, image analysis, and symmetric point clouds (see \cite{Carlsson09,CIdeSZ08,MKKWP19}). Such symmetries naturally lead to filtrations of $G$-spaces rather than ordinary topological spaces. 
From a mathematical point of view, equivariant homotopy theory is useful for studying spaces with group actions while keeping track of the information coming from the fixed-point spaces of all subgroups of $G$ (see \cite{May96,LMSM86}). This viewpoint has also started to appear in topological data analysis, with recent work on $G$-invariant persistent homology, persistent equivariant cohomology, and group-equivariant non-expansive operators (see \cite{Forsini15,BFGQ19,ALMSS24}). However, despite the recent progress in persistent homotopy theory, its equivariant version is still largely unexplored.

Classical persistent homology is generally insensitive to prescribed group actions and may fail to distinguish datasets that differ only through their symmetry structure. 
Equivariant version of persistent homology was first introduced by Frosini \cite{Forsini15}. Recently in \cite{ALMSS24}, the authors studied the persistent equivariant Borel cohomology and calculated an example for circle.  
Another recent development is the study of Group Equivariant Non-Expansive Operators (GENEOs), which shows how considering group symmetries into topological data analysis can improve both our theoretical understanding and practical approaches to data analysis and machine learning \cite{BFGQ19}. 
However, despite these developments, there are some gaps in the homotopy-theoretic approach to persistence that systematically takes group actions into consideration. In \cite{LM26}, the authors introduced an equivariant notion of \emph{homotopy type distance} from \cite{FLM19}. In our study, we explore the equivariant notions of \emph{homotopy interleaving distance} which was introduced in \cite{BL23, LS23}.

In this paper we develop some results of equivariant persistent homotopy theory generalizing the existing non-equivariant versions from \cite{BL23,LS23}. Throughout the paper, $G$ will denote a finite group. The equivariant persistent objects are regarded as $\I$-diagrams of $G$-spaces (or $G$-simplicial sets) equipped with the projective model structure, for some small category $\I$. When $\I = \RR$ we denote the \emph{persistent G-spaces} (indexed by $\RR$) by $\GTR$. Working in this setting allows us to combine techniques from equivariant homotopy theory with the categorical viewpoint underlying persistent topology.

Recall from \cite{CCGGO09} that for $\delta \geq 0$ two $\RR$-diagrams $X,Y \colon \RR \to \Top$ are called $\delta$-\emph{interleaved} if there exist natural transformations $\phi$ and $\psi$ satisfying the following two diagrams
    $$\begin{tikzcd}[column sep=1.5em]
        X_{\bullet} \arrow{rr}{\text{structure map}} \arrow[swap]{dr}{\phi_{\bullet}}  && X_{\bullet+2\delta} && Y_{\bullet} \arrow{rr}{\text{structure map}} \arrow[swap]{dr}{\psi_{\bullet}}  && Y_{\bullet+2\delta}\\
         & Y_{\bullet+\delta} \arrow[swap]{ur}{\psi_{\bullet+\delta}} &&&& X_{\bullet+\delta} \arrow[swap]{ur}{\phi_{\bullet+\delta}}
    \end{tikzcd}$$
commute. The \emph{interleaving distance} between $X$ and $Y$ is defined by 
    $$d_I(X,Y) \coloneq \inf \{ \delta \geq 0 \mid X \text{ and } Y \text{ are } \delta\text{-interleaved}\} \cup \{\infty\}.$$
The interleaving distance is an extended pseudometric which is not invariant under weak homotopy equivalence. 
For example, let $X,Y\colon \mathbb{R}\to\Top$ be the constant persistence spaces with $X(r)=[0,1]$ and $Y(r)=*$ for all $r\in\mathbb{R}$. Since the interval $[0,1]$ is contractible, the projection $[0,1]\to *$ is a weak homotopy equivalence, so $X$ and $Y$ are objectwise weakly equivalent. However, $X$ and $Y$ are not $0$-interleaved, as this would require $[0,1]$ and $*$ to be homeomorphic. Thus, the ordinary interleaving distance depends on the chosen model rather than the homotopy type. 
In \cite{BL23}, the authors introduced \emph{homotopy interleaving distance} on persistent spaces to overcome this limitation. Moreover, they proved that it is universal among all stable and homotopy invariant distances \cite[Definition 1.7]{BL23}. Their work established homotopy interleavings as the natural homotopical analogue of the classical interleaving distance and laid the foundation for studying persistent spaces through the methods of model categories and homotopy theory. 
In our work, we define the notions of \emph{$G$-stable} and \emph{$G$-homotopy invariant distance} (see Definition \ref{Def_stable and homotopy inv}). We introduce the \emph{$G$-homotopy interleaving distance} $d^G_{\HI}$ on the category $\GTR$ (see Definition \ref{Def_homotopy interleaving}) and prove the following. 

\begin{thma}[Proposition \ref{Prop_GH stable}, Theorem \ref{Th_Universality}]
    The $G$-homotopy interleaving distance $d^G_{\HI}$ is $G$-stable and $G$-homotopy invariant. Moreover, if $d^G$ is any $G$-stable and $G$-homotopy invariant distance, then $d^G$ is bounded above by $d^G_{\HI}$.
\end{thma}

The homotopy interleaving distance introduced by Blumberg and Lesnick provides a homotopy-theoretic notion of `\emph{approximate weak equivalence}' for filtrations. This motivates the investigation of other fundamental results from classical homotopy theory admit persistent analogues. So they conjectured a persistent analogue of the classical Whitehead theorem \cite[Conjecture 8.7]{BL23}. 

A direct approach to this conjecture using the homotopy interleaving distance is technically challenging, since homotopy interleavings are defined through homotopy coherent diagrams rather than morphisms in the homotopy category. To overcome this difficulty, Lanari and Scoccola \cite{LS23} considered the interleaving distance in the homotopy category, which admits a more tractable categorical description. They established a persistent Whitehead theorem relating interleavings of persistent homotopy groups to interleavings in the homotopy category. Motivated by these works, we define $G$-interleaving distance in homotopy category (see Definition \ref{Defn_GIHC}) to investigate the equivariant setting and establish the following equivariant analogue of persistent Whitehead theorem.

\begin{thmb}[Theorem \ref{Th_interleaving in homotopy category}]
    Let $X, Y \in G$-$\mathbb{S}^{\RR}$ be persistent spaces that are assumed to be projective cofibrant and $d$-skeletal if $G$-$\mathbb{S}=\GsS$, or persistent $G$-$CW$-complexes of dimension at most $d$ if $G$-$\mathbb{S}=\GTop$. 
    Let $\delta \geq 0$ be a real number. If there exists a morphism in the homotopy category $f \colon X \to Y(\delta) \in \mbox{Ho}(G$-$\mathbb{S}^{\RR})$ that induces $\delta$-interleavings in all homotopy groups, then $X$ and $Y$ are $(4(d+1)\delta)$-interleaved in the homotopy category. 
\end{thmb}

Let $T$ be a topological space and $\mathcal{U}=\{U_i\}_{i\in \Lambda}$ is an open cover of $T$. The \emph{nerve} of $\mathcal{U}$, denoted by $N(\mathcal{U})$, is the abstract simplicial complex whose vertex set is $\Lambda$, and whose simplices are the finite subsets $\{i_0,\ldots,i_k\}\subseteq \Lambda$ such that
    $$U_{i_0}\cap U_{i_1}\cap\cdots\cap U_{i_k}\neq \emptyset.$$
An open cover is called \emph{good} if every non-empty finite intersection of sets from $\mathcal{U}$ is contractible.
Nerve theorem is a classical result in homotopy theory which says a paracompact space $T$ is homotopy equivalent to the geometric realization of $\nrv(\mathcal{U})$ for any good cover $\mathcal{U}$ of $T$. 
A persistent version of the nerve lemma first appears in Chazal--Oudot\cite{CO08}. Afterwords, many articles explored the result, for example in \cite{BKRR23}, Bauer--Kerber--Roll--Rolle prove it along with several other versions useful for persistent setup.
On the other hand, one can find some versions of equivariant nerve lemma in \cite{HH13} and \cite{Yang14}, but Gonz\'alez--Gonz\'alez\cite{GG24} used a more topological approach. This motivates us to explore equivariant persistent nerve lemma. 

\begin{thmc}[Theorem \ref{Th_EPNL}]
Let $\I$ be a small category and $X\in \GTI$. If $\mathcal{U}$ is an equivariant good cover of $X$, then $X$ and $|\Nrv(\mathcal{U})|$ are weakly equivalent in the diagram category $\GTI$.
\end{thmc}

Using Theorem C,  we can deal with a fundamental question in topological data analysis: how the persistent homology of a random sample behaves as the sample size grows? This we can interpret as the \emph{weak law of large numbers} in TDA. In the non-equivariant setting, Blumberg-Lesnick proved that for i.i.d. samples from a compact metric space with a probability measure, the persistence space of \v{C}ech filtration converges in probability to a deterministic limit as we increase the sample size (\cite[Proposition 6.13]{BL23}). The proof uses a nerve-theoretic reduction to a combinatorial problem. Here we obtain an equivariant analogue of this result using Theorem C (see Proposition \ref{Prop_weak law}).

\vspace{.25cm}
\noindent \textbf{Organization:} In Section 2, we recall the necessary background on model categories, diagram categories, homotopy colimits, and homotopy interleavings. In Section 3, we introduce the G-homotopy interleaving distance and establish its universality. Section 4 is devoted to the equivariant persistent Whitehead theorem. In Section 5, we develop the equivariant persistent nerve theorem together with its application in proving the equivariant weak law of large numbers.


\section{Preliminaries}

This section is dedicated to recalling all the required backgrounds such as model category, homotopy category, diagram category, colimit, and homotopy colimit (\cite{Quillen67, Hovey, Hir03}). Since we are interested in diagrams of $G$-spaces, we revisit the projective model structure on it. We conclude this section by recollecting the ideas of interleaving categories, homotopy interleaving and interleaving in homotopy category (see \cite{BL23, LS23}).

\begin{mysubsection}{Model category} 
Let $\C$ be a category. Recall that a \emph{model structure} on $\C$ consists of three distinguished classes of morphisms: \emph{weak equivalences} ($\simeq$), \emph{fibrations} ($\twoheadrightarrow$) and \emph{cofibrations} ($\hookrightarrow$) satisfying the four axioms: $(i)$ \emph{two out of three property} for weak equivalences, $(ii)$ morphisms are \emph{closed under retract}, $(iii)$ \emph{lifting property}, and $(iv)$ \emph{factorization} of any morphism.  We call a (co)fibration \emph{acyclic} if it is also a weak equivalence.
A \emph{model category} is a complete and cocomplete category (i.e. all small limits and colimits exist) $\mathcal{C}$ with a model structure.
\end{mysubsection}

\begin{mysubsection}{Homotopy category} 
Associated to every model category $\C$, one can construct a \emph{homotopy category} $\Ho(\C)$ as follows. The homotopy category $\Ho(\C)$ is the category having the same objects as $\C$ and with morphisms
        $$\operatorname{Hom}_{\Ho(\C)}(X,Y): =\operatorname{\pi}(RQX,RQY)$$
    where $RQX,RQY$ cofibrant ($Q$) and fibrant ($R$) replacement of $X$ and $Y$ respectively and $\operatorname{\pi}(RQX,RQY)$ is the homotopy class of maps between them, see \cite [Theorem 1.1]{Quillen67}.
This association defines a functor $\gamma \colon \C \to \Ho(\C)$ from a given model category $\C$ to its corresponding homotopy category $\Ho(\C)$. In fact,$\gamma$ is a \emph{localization} of $\C$ with respect to the class of weak equivalences $W$ (see \cite[Theorem 6.2]{DWJS95}). Hence, $\gamma$ sends weak equivalences to isomorphisms and for any functor $F\colon \C \to \mathcal{D}$ with the same property, there exists a unique functor $G\colon \Ho(\C)\to \mathcal{D}$ such that the following diagram commutes.
        $$\xymatrix{
        \C \ar[rr]^F \ar[d]_{\gamma} && \mathcal{D}\\
        \Ho(\C) \ar[rru]_G  &&
        }$$
Let $X,Y \in \operatorname{ob}(\C)$.  
We say $X$ and $Y$ are \emph{weakly-equivalent} and write $X \simeq Y$ if $X,Y\in  \operatorname{ob}(\C)$ are isomorphic in $\Ho(\C)$. Since $\gamma$ is a localization, it follows that $X \simeq Y$ if and only if there is a zigzag of weak-equivalences in $\C$ connecting $X$ and $Y$.
    $$\xymatrix{
        & W_1 \ar[rd]^{\simeq} \ar[ld]_{\simeq} && \dots \ar[rd]^{\simeq} \ar[ld]_{\simeq} && W_n \ar[rd]^{\simeq} \ar[ld]_{\simeq} &\\
        X && W_2 && W_{n-1} && Y
    }$$
A filtration is precisely a functor from a poset, viewed as a category, to a target category. Thus, persistent objects are naturally studied in the corresponding diagram category.
\end{mysubsection}

\begin{mysubsection}{Diagram Category}\label{Diagram Category} 
Let $\C$ be a category and $\I$ a small category. The \emph{diagram category} $\C^{\I}$ is the category whose objects are functors $X \colon \I \to \C$ and whose morphisms are natural transformations between such functors. For each object $t \in \I$, we denote the object $X(t)$ of the diagram $X$ by $X_t$. We may use both notions interchangeably.

Consider $X,Y\in \operatorname{ob}(\C^\I)$. Then, a natural transformation $\alpha\colon X \to Y$ is an \emph{object-wise weak-equivalence} if $\alpha_t\colon X_t \to Y_t$ is weak-equivalence for all $t \in \operatorname{ob}(\I)$. Similarly, $\alpha $ is object-wise fibration and cofibration if $\alpha_t\colon X_t \to Y_t$ is fibration and cofibration, respectively.

 By the \emph{projective model structure} is e on the diagram category $\C^{\I}$ we mean in which weak-equivalence and fibrations are object-wise. In particular, if $\C$ is a \emph{cofibrantly-generated model category}  \cite [Definition 11.1.2]{Hir03}, the diagram category $\C^\I$ is cofibrantly generated and admits a projective model structure (see \cite [Theorem 11.6.1]{Hir03}). Furthermore, let $\mathcal{I}$ and $\mathcal{J}$ denote the sets of generating cofibrations and generating trivial cofibrations, respectively, for the cofibrantly generated model structure on $\C$. Then the projective model structure of $\C^\I$ is cofibrantly generated with generating cofibrations and generating trivial cofibrations given by 
    $$\{{\I}(i,-)\odot f\colon i \in \I,f \in \mathcal{I}\} \quad \text{and} \quad \{\I(i,-)\odot g\colon i \in \I,g \in \mathcal{J}\},$$
respectively. Here ${\I}(i,-)\colon \I \to \text{Set}$ is a functor. Let $A\in \C$ be the functor ${\I}(i,-)\odot A\colon \I \to \C$ is defined by an object $j\mapsto \coprod_{\I(i,j)}A$ and a morphism $u\colon \alpha \to \beta \in \text{morphism}(\I)$:
    $$(\I(i,-)\odot A)(u) \colon \coprod_{\I(i,\alpha)}A \longrightarrow \coprod_{\I(i,\beta)}A$$
is defined uniquely by specifying where each component maps to. In particular, label a copy of $A$ in the domain, by a morphism $v\colon i\to \alpha$. Then the copy of $X$ maps to the copy of $A$ corresponding to $u\circ v$ via $\text{id}_A \colon A\to A$. In formally it satisfies the relation
     $$(\I(i,-)\odot A)(u) \circ \iota_v = \iota_{u\circ v}$$
where $\iota_v$ is the inclusion of the copy of $A$ indexed by $v$. Hence,  $(\I(i,-)\odot A)$ is a diagram and this diagram is called a \emph{free diagram on $A$ generated at $i$} (cf. \cite [Definition 11.5.25]{Hir03}). Now if $f\colon A \to B$ be a morphism in $\C$, then $(\I(i,-)\odot f)$ becomes a natural transformation between two diagrams $(\I(i,-)\odot A)$ and $(\I(i,-)\odot B)$. For each $d\in \I$ the component of the natural transformation at $d$ is:
    $$(\I(i,-)\odot f)_d= \coprod_{\I(i,d)}f\colon \coprod_{\I(i,d)}A \longrightarrow \coprod_{\I(i,d)}B$$
equivalently if $\iota^A_v \colon A \to \coprod_{\I(i,d)} A$ denotes the inclusion of the copy indexed by a morphism $v\colon i \to d$, then the component is uniquely determined by: 
    $$(\I(i,-)\odot f)_d \circ \iota^A_v = \iota^B_v \circ f$$
for every $v \in \I(i,d)$. For further details, see \cite [Section 11.6]{Hir03}. For simplicity we write  $i\odot A$ in place of $\I(i,-)\odot A$. 

We are particularly interested in the projective model structure when the indexing category is a poset $(P, \leq)$. In this case, for $r\in P$ and $A\in \C$, the functor $r\odot A$ is given by,
    $$r\odot A(t)=
        \begin{cases}
            {A \text{  if } t\geq r}\\
            \emptyset \text{ otherwise}
        \end{cases}$$
where the non-trivial structure morphisms are all given by the identity morphism $\text{id}_A$.        

\begin{remark}
Note that, two diagrams $X, Y \in \operatorname{ob}(\C^\I)$ are isomorphic in the homotopy category $\operatorname{Ho}(\C^\I)$, or they are weakly-equivalent in the projective model structure on $\C^\I$, if and only if there is a zigzag of object-wise weak-equivalences.
\end{remark}

Having introduced diagram categories in a general setting, we now specialize to diagrams of $G$-spaces. In TDA, the persistent objects coming from the data are modeled as a filtration of topological spaces. Here, instead of taking only topological spaces, we take a filtration of $G$-topological spaces indexed by a poset $P$. Our objective is to combine equivariant homotopy theory with persistence, and to this end we study an appropriate model structure on $\GTop^P$ in the following section.
\end{mysubsection}


\begin{mysubsection}{Projective model structure on $\GTop^{\I}$} 
The category $\Top$ of compactly generated weakly Hausdorff (CGWH) spaces is cofibrantly generated with respect to the Serre model structure. Its generating cofibrations are the boundary inclusions
    $$S^{n-1}\hookrightarrow D^n,\qquad n\geq 0.$$
Similarly, the category $\GTop$ of CGWH spaces equipped with a continuous $G$-action admits two standard model structures: the Borel model structure and the Bredon model structure. Throughout this paper, we assume that $G$ is a finite group.

In the \emph{Borel model structure}, a map is a weak equivalence (respectively, fibration) if and only if the underlying map of topological spaces, obtained by forgetting the $G$-action, is a weak homotopy equivalence (respectively, a Serre fibration).

In contrast, the \emph{Bredon model structure} is defined in terms of fixed-point spaces. A map $f\colon T_1 \to T_2$ is a weak equivalence if, for every subgroup $H\leq G$, the induced map
$f^H: T_1^H \to T_2^H$ is a weak homotopy equivalence. Likewise, $f$ is a fibration if, for every subgroup $H\leq G$, the induced map $f^H: T_1^H \to T_2^H$ is a Serre fibration. 

Both model structures are cofibrantly generated. The generating cofibrations and the generating trivial cofibrations for the Borel model structure are respectively:
    $$G \times S^{n-1} \to G \times D^n \quad \text{and} \quad G\times \emptyset \to G\times D^n \quad \forall n\geq 0.$$
The generating cofibrations and the generating trivial cofibrations for the Bredon model structure are respectively:
    $$G/H \times S^{n-1} \to G/H \times D^n \quad \text{and} \quad G/H\times \emptyset \to G/H\times D^n \quad \forall~ H\leq G \quad \text{and} \quad n\geq 0.$$
Henceforth, we use the Bredon model structure as a model structure on $\GTop$. The diagram category $\GTop^\I$ for a small category $\I$ admits a projective model structure, and its generating cofibrations and generating trivial fibrations are respectively:
    $$r\odot(G/H \times S^{n-1}) \to r\odot(G/H \times D^n) \quad \text{and} \quad r\odot (G/H\times \emptyset) \to r\odot(G/H\times D^n)$$
$\forall~ H\leq G$, $n\geq 0$ and $r \in \I $, see Subsection \ref{Diagram Category}.

Persistent objects are modeled as diagrams. Hence, it is often desirable to combine the objects at all filtration levels into a single space that reflects the entire diagram. The categorical construction that performs this gluing is the colimit. However, since ordinary colimits do not generally preserve homotopy-theoretic information, we also consider their homotopy-invariant analogue, the homotopy colimit, in the following section. 
\end{mysubsection}


\begin{mysubsection}{Colimit and Homotopy colimit}\label{Subsec_blowup bar and hcolim} For a given diagram of spaces \emph{colimit} provides a universal way to glue them together and give a new object in the underlying category. In particular, a colimit of the diagram $X \in \GTop^{\I}$, denoted by $\colim~X$, is a $G$-space together with a family of $G$-maps $\iota_{t} \colon X_t \to \colim~X$ such that
    \begin{enumerate}[(i)]
        \item for every morphism $\alpha \colon t \to s$ in $\I$, $\iota_s \circ X(\alpha)=\iota_t$,
        \item for every $G$-space $Y$ and every cocone $\{f_t:X(t)\to Y\}_{t \in \I}$, consisting of $G$-maps satisfying $f_s \circ X(\alpha)=f_t$ for every morphism $\alpha:t\to s$, there exists a unique $G$-equivariant map $f:\colim~X \longrightarrow Y$ such that $    f\circ\iota_t=f_t$ for every $t\in \I$.
    \end{enumerate}

The colimit operation, in general, does not preserve weak equivalences. Let $F, F'\colon \I \to \C$ be two diagrams with object-wise weak equivalence, $F(i)\xrightarrow{\simeq} F'(i)$ for all $i\in \I$. Then it is not true in general that $\colim~F \xrightarrow{\simeq} \text{colim} F'$. For example, let $\I$ to be the category $\{\bullet \leftarrow \bullet \rightarrow \bullet \}$.
Define two $\I$-shaped diagrams $F,F' \in \Top^{\I}$ by
    $$F\colon \{ D^n \hookleftarrow S^{n-1} \hookrightarrow D^n\}, \quad  F'\colon \{ * \hookleftarrow S^{n-1} \hookrightarrow * \}.$$
Consider the commutative diagram
       $$
        \xymatrix{
         D^n   \ar[d]_{\simeq} & & S^{n-1} \ar@{_{(}->}[ll] \ar@{^{(}->}[rr]\ar[d]^{Id}  & & D^n \ar[d]^{\simeq}\\
          \ast   & & S^{n-1} \ar[ll] \ar[rr] & & \ast.
        }
      $$  
Here, the top horizontal maps are the boundary inclusions. The vertical maps define a natural transformation $F \implies F'$, that is an object-wise weak-equivalence. Moreover, $\colim~F$ is $S^n$, while $\colim~F'$ is $\ast$. Hence, the induced map $\colim~F \to \colim~F'$ is not a weak equivalence. This example motivates the introduction of homotopy colimit. It provides a homotopy-invariant replacement for ordinary colimit. In particular, objectwise weakly-equivalent diagrams have weakly-equivalent homotopy colimits. While the ordinary colimit is defined by a universal property in the underlying category, the homotopy colimit is defined by the corresponding universal property in the homotopy category via the total left derived functor of the colimit functor. 
Let $\C$ be a model category and $\colim\colon \C^\I \to \C$ a functor, The homotopy colimit
    $$\text{hocolim}\colon \Ho(\C^\I) \to \Ho(\C)$$
is a left-derived functor of $\colim$, provided it exists.

In practice, computing homotopy colimits directly from the definition can be difficult. This motivates the use of explicit constructions that provide concrete models for homotopy colimits. One such construction is the bar construction, which encodes the objects and morphisms of a diagram into a simplicial object. Under suitable hypotheses, the geometric realization of this simplicial object provides a model for the homotopy colimit which we use in Section \ref{Sec_EPNL}. 

\end{mysubsection}


\begin{mysubsection}{Homotopy Interleaving}\label{Subsec_Homotopy Interleaving}
The idea of interleaving distance was initially studied for persistent homology. Though it is natural to define these notions for $\RR$-spaces, but the interleaving distance is not a homotopy invariant in $\RR$-spaces. In \cite{BL23}, Blumberg and Lesnick defined a homotopical generalisation of the interleaving distance. In this subsection, we revisit the background and definitions following \cite{BL23}.

We call a category $\mathcal{C}$ \emph{thin category} if for any two objects $a,b \in \mathcal{C}$, there exists at most one morphism $f \colon a \to b$ between them. Given a thin category $\mathcal{C}$ with a morphism $f\colon a \to b$ in $\mathcal{C}$ and a functor $F \colon \mathcal{C} \to \mathcal{D}$, we denote $F(f)$ as $F_{a,b}$.

Let $\delta \geq 0$ be a real number. The $\delta$-\emph{interleaving category} is a thin category with object set $\RR \times \{0,1\}$ and morphism $(r,\alpha) \to (s,\beta)$ if and only if either (i) $r+\delta \leq s$ or (ii) $\alpha=\beta$ and $r \leq s$.
We denote this category by $\mathcal{I}^{\delta}$. We now can define two functors $E^0, E^1 \colon \RR \to \mathcal{I}^{\delta}$ by $E^0(r)=(r,0)$ and $E^1(r)=(r,1)$.

\begin{figure}[h]
    \begin{tikzpicture}[scale=.7]
        \draw[very thick,white!30!red] (-10,1)--(10,1);
        \draw[very thick,cyan] (-10,-1)--(10,-1);

        \node[right,red] at (10,1) {$\RR \times \{0\}$};
        \node[right,cyan] at (10,-1) {$\RR \times \{1\}$};

        \node at (3,1) {$\bullet$};
        \node[above] at (3,1) {$(r_2,0)$};
        \node at (3,-1) {$\bullet$};
        \node[below] at (3,-1) {$(r_2,1)$};
        \node at (7,1) {$\bullet$};
        \node[above] at (7,1) {$(s_2,0)$};
        \node at (7,-1) {$\bullet$};
        \node[below] at (7,-1) {$(s_2,1)$};

        \draw[] (3,1)--(7,1) node[midway, sloped] {$>$};
        \draw[] (3,-1)--(7,-1) node[midway, sloped] {$>$};

        \node at (-7,1) {$\bullet$};
        \node[above] at (-7,1) {$(r_1,0)$};
        \node at (-5,-1) {$\bullet$};
        \node[below] at (-5,-1) {$(s_1,1)$};
        \node at (-3,1) {$\bullet$};
        \node[above] at (-3,1) {$(r'_1,0)$};
        
        \draw[] (-7,1)--(-5,-1) node[midway, sloped] {$>$};
        \draw[] (-5,-1)--(-3,1) node[midway, sloped] {$>$};
        \draw[] (-7,1)--(-3,1) node[midway, sloped] {$>$};

        \node[left] at (-6.5,0) {$r_1+\delta \leq s_1$};
        \node[right] at (-3.5,0) {$s_1 +\delta \leq r_1'$};
    \end{tikzpicture}
    \caption{$\mathcal{I}^{\delta} \colon \delta$-interleaving category.}
\end{figure}
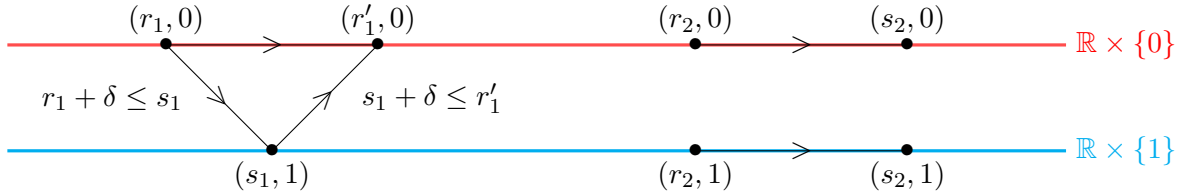
\noindent For any category $\mathcal{C}$ and two functors $X, Y \colon \RR \to \mathcal{C}$, the $\delta$-\emph{interleaving} between $X$ and $Y$ is defined to be a functor $Z \colon \mathcal{I}^{\delta} \to \mathcal{C}$ such that $Z \circ E^0=X$ and $Z \circ E^1=Y$. If such $Z$ exists for some $\delta \geq 0$, we say $X$ and $Y$ are $\delta$-interleaved.

Now for a diagram $X$, we define the \emph{shift functor} as follows
    \begin{myeq}\label{Eq_shift functor}
        X(\delta) \colon \RR \to \mathcal{C} \quad \quad \text{defined by } X(\delta)_r \coloneq X_{r+\delta} \quad \text{and} \quad X(\delta)_{r,s} \coloneq X_{r+\delta, s+\delta}.
    \end{myeq}  
Though in practice we use the term shift functor, notice that it is a left shift. Dually, a $\epsilon$-\emph{shift to the right} is defined by
    $$\epsilon \cdot X \colon \RR \to \mathcal{C} \quad \quad \text{defined by } \epsilon \cdot X_r \coloneq X_{r-\epsilon} \quad \text{and} \quad \epsilon \cdot X_{r,s} \coloneq X_{r-\epsilon, s-\epsilon}.$$
The natural transformations $f \colon X \to Y(\delta)$ are called $\delta$-\emph{morphisms}, and often denoted by $f \colon X \to_{\delta} Y$. Since we have the natural bijections
    \begin{myeq}\label{Eq_epsilon_delta}
        \hom(\epsilon \cdot X, Y(\delta)) \cong \hom(X, Y(\epsilon+\delta)) \cong \hom ((\epsilon+\delta) \cdot X, Y),
    \end{myeq}
a $\delta$-morphism $f \colon X \to_{\delta} Y$ can also be treated as $f \colon \delta \cdot X \to Y$.

From our construction, $Z$ restricts to a pair of $\delta$-morphisms $X \to Y(\delta)$ and $Y \to X(\delta)$, and conversely these morphisms fully determine $Z$. 
These morphisms are called $\delta$-\emph{interleaving morphisms}. In particular, if $\delta=0$, they are an inverse pair of morphisms between $X$ and $Y$. The \emph{interleaving distance} between $X$ and $Y$ is a function $d_I \colon \text{ob} (\mathcal{C}^{\RR}) \times \text{ob} (\mathcal{C}^{\RR}) \to [0, \infty]$, which is defined by
    $$d_I(X, Y) \coloneq \inf \{ \delta \colon X \text{ and } Y \text{ are } \delta \text{-interleaved}  \}.$$
Lanari--Scoccola\cite{LS23} introduced a generalized notion of $\delta$-interleaving by $(\epsilon, \delta)$-interleaving. Given two diagrams $X$ and $Y$ are $(\epsilon, \delta)$-interleaved if there exists $\epsilon$-morphism $X \to Y(\epsilon)$ and $\delta$-morphism $Y \to X(\delta)$ satisfying the following two commutative diagrams. 
    $$\begin{tikzcd}[column sep=1.5em]
        X_{\bullet} \arrow{rr}{\text{structure map}} \arrow[swap]{dr}{\phi_{\bullet}}  && X_{\bullet+\epsilon+\delta} && Y_{\bullet} \arrow{rr}{\text{structure map}} \arrow[swap]{dr}{\psi_{\bullet}}  && Y_{\bullet+\epsilon+\delta}\\
         & Y_{\bullet+\epsilon} \arrow[swap]{ur}{\psi_{\bullet+\epsilon}} &&&& X_{\bullet+\delta} \arrow[swap]{ur}{\phi_{\bullet+\delta}}
    \end{tikzcd}$$

\begin{remark}\label{Rem_delta to 2delta}
    A $\delta$-morphism $p \colon X \to_{\delta} Y$ induces a $\delta$-interleaving means that there exists another $\delta$-morphism $q \colon Y \to_{\delta} X$ such that $q \circ p=\text{structure map of }X$ and $p \circ q=\text{structure map of }Y$. Observe that a $\delta$-interleaving is a $(0,2\delta)$-interleaving by the natural bijection in \eqref{Eq_epsilon_delta}.
\end{remark}

If $X,Y$ are $\delta$-interleaved and $Y,W$ are $\epsilon$-interleaved, then $X$ and $W$ are $(\delta+\epsilon)$-interleaved, i.e., $d_I$ defines a distance on ob$(\mathcal{C}^{\RR})$. 
Also for $X \cong X'$ in ob$(\mathcal{C}^{\RR})$, $d_I(X,Y)=d_I(X',Y)$ for all $Y \in \text{ob}(\mathcal{C}^{\RR})$.
Notice that the distance $d_I$ on the category of $\RR$-spaces is not a homotopy invariant. For example, take two $\RR$-spaces $X$ and $Y$ defined by $X_r=0$ and $Y_r=\RR$ for all $r \in \RR$. In \cite{BL23}, the authors introduced a homotopical generalization of the interleaving distance between diagrams as follows. 

\begin{definition}\label{def_hinterleaved}
    For $\delta \geq 0$, two $\RR$-spaces $X$ and $Y$ are called 
    \begin{enumerate}
        \item $\delta$-\emph{homotopy interleaved} if there exist $\RR$-spaces $X', Y'$ and weak equivalences $X' \simeq X, ~~Y' \simeq Y$ such that $X'$ and $Y'$ are $\delta$-interleaved. The \emph{homotopy interleaving distance} between two $\RR$-spaces is defined by 
            $$d_{\HI}(X,Y) \coloneq \inf \{\delta\geq 0 \colon X \text{ and } Y \text{ are } \delta \text{-homotopy interleaved}\}.$$

        \item $\delta$-\emph{homotopy commutative interleaved} if the composites $\RR\xrightarrow[]{X,Y} \Top\xrightarrow[]{\gamma} \Ho(\Top)$ are $\delta$-interleaved. The \emph{homotopy commutative interleaving distance} between two $\RR$-spaces is defined by 
            $$d_{\HC}(X,Y) \coloneq \inf \{\delta\geq 0 \colon X \text{ and } Y \text{ are } \delta \text{-homotopy commutative interleaved}\}.$$

        \item $\delta$-\emph{interleaved in homotopy category} if they are $\delta$-interleaved as objects of $\Ho(\Top^{\RR})$. The \emph{interleaving distance between two spaces in the homotopy category} is 
            $$d_{\IHC}(X,Y) \coloneq \inf \{\delta\geq 0 \colon X \text{ and } Y \text{ are } \delta \text{-interleaved in homotopy category}\}.$$
    \end{enumerate}
\noindent In the above three definitions, if no such $\delta $ exists, we say $X$ and $Y$ are $\infty$-interleaved in the respective definition.
\end{definition}

\end{mysubsection}


\section{$G$-Homotopy Interleaving}

In \cite{BL23}, the authors introduced the homotopy interleaving distance and proved that it is universal among all stable and homotopy invariant distances on persistent spaces. 
In our work, we introduced the equivariant notions such as $G$-stable and $G$-homotopy invariant distance, and $G$-homotopy interleaving distance $d^G_{\HI}$. We show that $d^G_{\HI}$ is a $G$-stable and $G$-homotopy invariant distance. Finally, we prove the universality of $G$-homotopy interleaving distance among any $G$-stable and $G$-homotopy invariant distance. Along the way, we also prove an equivariant version of topological stability in Proposition \ref{Prop_GH stable} analogous to that of non-equivariant version \cite[Proposition 1.8]{BL23}.

\begin{mysubsection}{$G$-stability and $G$-homotopy invariance}\label{G-stability}
This subsection is devoted to recall some definitions and introduce some notions for future use. First, we recall the definition of the equivariant Gromov--Hausdorff distance between two metric spaces from \cite{LM26}.

A \emph{correspondence} between two sets $M$ and $N$ is any surjective relation $R \subseteq M \times N$, and the set of all correspondence between $M$ and $N$ is denoted by $\mathcal{R}(M, N)$, see \cite{BBI01}. By surjective relation, we mean the coordinate projections $\mbox{proj}_M \colon R \to M$ and $\mbox{proj}_N \colon R \to N$ are surjections.
Now let $M$ and $N$ be two $G$-sets. A correspondence $R$ between $M$ and $N$ is called $G$-\emph{correspondence} if $(g.m, g.n) \in R$ for all $(m,n) \in R \subseteq M \times N$ and for all $g \in G$. We denote any $G$-correspondence between $M$ and $N$ by $R_G$ and the set of all $G$-correspondences between $M$ and $N$ by $\mathcal{R}_G(M,N)$.
Now let $(M, d_M)$ and $(N, d_N)$ be two bounded metric spaces. For any non-empty relation $R \subseteq M \times N$, its distortion is defined by
    $$\mbox{dis} (R) \coloneq \sup_{(m,n), (m',n') \in R} |d_M(m,m')-d_N(n,n')|.$$

\begin{definition}[{\cite[Definition 10]{LM26}}]
    For any two $G$-metric spaces $M$ and $N$, the $G$-\emph{Gromov--Hausdorff distance} is defined as 
        \begin{myeq}\label{Eq_GH distance}
            d^G_{\GH}(M,N) \coloneq \frac{1}{2} \inf \{ \mbox{dis}(R) \colon R \in \mathcal{R}_G(M,N)\}.
        \end{myeq}
\end{definition}

Now we define $G$-stable and $G$-homotopy invariant distance. We show in Proposition \ref{Prop_GH stable} that any $G$-stable and $G$-homotopy invariant distance is bounded above by $d^G_{\GH}$.

Let $T$ be a $G$-topological space for a finite group $G$ and a $G$-equivariant function $\gamma \colon T \to \RR$, with trivial $G$-action on $\RR$. The \emph{sublevel filtration} $\mathscr{S}(\gamma) \colon \RR \to \GTop$ for $\gamma$ is defined by 
    $$\mathscr{S}(\gamma)_{t} \coloneq \gamma^{-1} (-\infty, t],$$
for any $t\in \RR$. Note that $\gamma$ may not be continuous. For two $G$-functions $\gamma$ and $\lambda$, we define 
    $$d_{\infty}^G(\gamma,\lambda) \coloneq \sup _{x \in T} |\gamma(x)-\lambda(x)|.$$

Now we recall the notion of equivariant weak equivalence. Let $\I$ be any small category and $X, Y \colon \I \to \GTop$ be two functors. A natural transformation $\alpha \colon X \to Y$ is an (\emph{objectwise}) $G$-\emph{weak equivalence}, denoted by $X \xrightarrow{\simeq_G} Y$, if the $G$-map $\alpha_a \colon X_a \to Y_a$ is a weak homotopy equivalence for all $a \in \I$. We call $X$ and $Y$ are $G$-\emph{weakly equivalent} if there exists a zigzag of objectwise weak equivalences as follows
    $$\xymatrix{
    & W_1 \ar[rd]^{\simeq_G} \ar[ld]_{\simeq_G} && \dots \ar[rd]^{\simeq_G} \ar[ld]_{\simeq_G} && W_n \ar[rd]^{\simeq_G} \ar[ld]_{\simeq_G} &\\
    X && W_2 && W_{n-1} && Y
    }$$
for some $n$. The $G$-weak equivalence is denoted by $X \simeq_G Y$.

\begin{definition}\label{Def_stable and homotopy inv}
    Let $d^G$ be a distance on $\GTop^{\RR}$ for some finite group $G$. We call the distance $d^G$ is $G$-\emph{stable} if for any $T \in \mbox{ob} (\GTop)$ and $G$-maps $\gamma, \lambda \colon T \to \RR$, 
        \begin{myeq}\label{Eq_G stable}
            d^G(\mathscr{S}(\gamma), \mathscr{S}(\lambda)) \leq d^G_{\infty}(\gamma,\lambda).
        \end{myeq}
    We say $d^G$ is $G$-\emph{homotopy invariant} if for any $X, Y \in \GTop^{\RR}$,
        \begin{myeq}\label{Eq_G homotopy invariant}
            d^G(X,Y)=0 \quad \text{whenever } X \simeq_{G} Y.
        \end{myeq}
\end{definition}

Now we prove the equivariant version of Quillen's Theorem A from \cite{Quillen72} which is an useful tool for the upcoming Preposition \ref{Prop_GH stable}.

\begin{lemma}[Equivariant Quillen's Theorem A]\label{Th_equiv QT}
    Let $S$ and $T$ be $G$-simplicial complexes, and let $f \colon S \to T$ be a $G$-simplicial map. Suppose that for every subgroup $H \leq G$ and every simplex $\sigma \subseteq T^H$, the inverse image $(f^H)^{-1}(\sigma)$ is contractible. Then $f$ is a $G$-homotopy equivalence.
\end{lemma}

\begin{proof}
    For each subgroup $H \leq G$, the induced map $f^H \colon S^H \to T^H$ is a simplicial map between simplicial complexes. By hypothesis, for every simplex $\sigma \subseteq T^H$, the inverse image $(f^H)^{-1}(\sigma)$ is contractible. Hence, by the classical Quillen Theorem~A the map $f^H$ is a homotopy equivalence (\cite{Quillen72}).    
    Since this holds for every subgroup $H \leq G$, the equivariant Whitehead theorem implies that $f$ is a $G$-homotopy equivalence (\cite[Theorem 5.3]{Matumoto71}) .
    
\end{proof}

\begin{proposition}\label{Prop_GH stable}
    Let $d^G$ be a $G$-stable and $G$-homotopy invariant distance on $\GTop^\RR$. Then for any two $G$-metric spaces $M$ and $N$, we have 
        \begin{myeq}\label{Eq_GH stable}
            d^G(\mathcal{R}(M),\mathcal{R}(N)) \leq d^G_{\GH} (M,N).
        \end{myeq}
\end{proposition}

\begin{proof}
    Let $d^G_{\GH}(M,N)< \delta$. Then there exists a $G$-correspondence $ R_G \subset M \times N$ such that 
        $$|d^G_M(m,m')-d^G_N(n,n')| \leq 2 \delta \quad \forall (m,n),(m',n') \in R_G.$$
    Let $[R_G]$ be the maximal simplicial complex with with vertices $R_G$. Note that $[R_G]$ consists of all non-empty finite subsets of $R_G$. Since the vertices in $[R_G]$ inherits $G$-action from the definition of $R_G$, $[R_G]$ is a $G$-simplicial complex.

    Now we define a simplicial filtration $X^M$ on $[R_G]$ by taking a $G$-simplex
        $$\sigma \in X^M_r \Longleftrightarrow d_M^G(\mbox{proj}_M(u), \mbox{proj}_M(v)) \leq 2r,$$
    for all $r \in \RR$ and for all $u,v \in \sigma$.
    Thus, $X^P$ induces a function $\gamma^M \colon [R_G] \to \RR$ which sends a $G$-simplex $\sigma$ to the minimum $r \in \RR$ such that $\sigma \in X^M_r$. Note that the simplicial filtration $X^M$ is equal to the simplicial sublevel filtration $\mathscr{S}(\gamma^M)$. A similar construction can be done for $N$ to define a simplicial filtration $X^N$ and an induced function $\gamma^N \colon [R_G] \to \RR$  such that $X^N$ is equal to the simplicial sublevel filtration $\mathscr{S}(\gamma^N)$. Thus, our construction leads us to
        \begin{myeq}\label{Eq_equivalence in simplicial filtration}
            X^M=\mathscr{S}(\gamma^M) \quad \text{and} \quad X^N=\mathscr{S}(\gamma^N).
        \end{myeq}
    Also, from our choice of $R_G$, we have $d^G_{\infty}(\gamma^M, \gamma^N) \leq \delta$. Then \eqref{Eq_equivalence in simplicial filtration} and the $G$-stability of $d^G$ (see \eqref{Eq_G stable}) implies $d^G(X^M, X^N) \leq \delta$.

    Note that projection $\mbox{prop}_M \colon R_G \to M$ induces a morphism $f \colon X^M \to \mathcal{R}(M)$. By Theorem \ref{Th_equiv QT} for $G$-simplicial complexes, $f$ is an objectwise $G$-homotopy equivalence. Similarly, $\mbox{prop}_N \colon R_G \to N$ induces another objectwise $G$-homotopy equivalence $g \colon X^N \to \mathcal{R}(N)$.
    Thus, by the $G$-homotopy invariance of $d^G$ (see \eqref{Eq_G homotopy invariant}), we have
        $$d^G(\mathcal{R}(M), X^M)=0=d^G(\mathcal{R}(N), X^N).$$
    Since, $d^G$ is distance, by triangle inequality, $d^G(\mathcal{R}(M),d^G\mathcal{R}(N)) \leq \delta$. As this inequality holds for any choice of $\delta$ satisfying $d^G_{\GH}(M,N) <\delta$, we get our desired result as described in \eqref{Eq_GH stable}.
    
\end{proof}

\end{mysubsection}

\begin{mysubsection}{$G$-homotopy interleaving}
The stability theorem proved above motivates the construction of a homotopy invariant distance on $\GTop^\RR$. To achieve this, we introduce the notion of a $G$-homotopy interleaving, which leads to the definition of the equivariant homotopy interleaving distance.

\begin{definition}[$G$-homotopy interleaving distance] \label{Def_homotopy interleaving}
    Let $X, Y \in \GTR$. We say $X, Y$ are $(G,\delta)$-\emph{homotopy interleaved} if there exists $X', Y'$ with $X\simeq_{G} X'$ and $Y \simeq_{G} Y'$ such that the following diagrams
    $$\begin{tikzcd}[column sep=1.5em]
        X'_{\bullet} \arrow{rr} \arrow[swap]{dr}{\phi_{\bullet}}  && X'_{\bullet+2\delta} && Y'_{\bullet} \arrow{rr} \arrow[swap]{dr}{\psi_{\bullet}}  && Y'_{\bullet+2\delta}\\
         & Y'_{\bullet+\delta} \arrow[swap]{ur}{\psi_{\bullet+\delta}} &&&& X'_{\bullet+\delta} \arrow[swap]{ur}{\phi_{\bullet+\delta}}
    \end{tikzcd}$$
are strictly commutative for $G$-equivariant maps $\phi\colon X\to Y$ and $\psi\colon Y \to X$. The $G$-\emph{homotopy interleaving distance} between $X$ and $Y$ are defined by 
    $$d_{\HI}^G (X,Y) \coloneq \inf \{ \delta \colon X \text{ and } Y \text{ are } (G, \delta) \text{-homotopy interleaved}\}.$$
\end{definition}

\begin{proposition}\label{Prop_stable and HID}
    With notations as above, the $G$-homotopy interleaving distance $d_{HI}^G$ is $G$-stable, $G$-homotopy invariant distance.
\end{proposition}

\begin{proof}
    From the definition, it follows that $d^G_{\HI}$ is symmetric, non-negative, and for any $T \in \GTop$, $d^G_{\HI}(T,T)=0$. Moreover, $d^G_{\HI}$ satisfies the triangle inequality with similar arguments as in \cite[Section 4]{BL23}. Thus, $d^G_{\HI}$ is a distance. One can also argue that $d^G_{\HI}$ is a distance by \cite[Proposition 2.3]{LS23} since $\GTop$ is a cofibrantly generated model category.
    
    By definition, $d_{\HI}^G$ is $G$-homotopy invariant. For stability we have to show that for any $T \in \GTop$ and $G$-maps $\lambda,\gamma\colon T \to \RR$,
        $$d_{\HI}^G(\mathscr{S}(\gamma),\mathscr{S}(\lambda)) \leq d_\infty(\gamma,\lambda).$$ 
    Let $d^G_\infty(\gamma,\lambda) = \text{sup}_{x\in T} |\gamma(x)-\lambda(x)| = \delta$. Suppose $y\in \mathscr{S}(\lambda)_r$, so that $\lambda(y)\leq r$. Hence, we have
        $$\ \gamma(y)\leq \lambda(y) + \delta \leq r+\delta \quad \implies \quad y \in (\gamma)^{-1}(-\infty,r+\delta] \quad \implies y \in \mathscr{S}(\gamma)_{r+\delta}.$$
    Since the $G$-action on $\RR$ is trivial, $\gamma(g.y)= \gamma(y)\leq r+\delta$ for all $g \in G$. Therefore, we have $y \in \mathscr{S}(\gamma)_{r+\delta}$ if and only if $g.y \in \mathscr{S}(\gamma)_{r+\delta}$ for all $g\in G$. Similarly, $g.y\in \mathscr{S}(\gamma)_{r+\delta}$ implies $g.y\in \mathscr{S}(\lambda)_{r+2\delta}$. Therefore, we get a $(G,\delta)$-interleaving. Thus, we have $d_{\HI}^G \leq \delta$. So we get our desired result.
    
\end{proof}

In the rest of the section, we focus on proving that the $G$-homotopy interleaving is $G$-stable. For that we first set the premise with some definitions and results.
For any small category $\I$, a functor $X \colon \I \to \GTop$ is called a \emph{closed filtration} if each of the internal maps $X_{a,b} \colon X_a \to X_b$ is a closed inclusion. By closed inclusion, we mean an injective and closed map.

\begin{proposition}
    For a small category $\I$, any cofibrant diagram in {$\GTI$} is a closed filtration.
\end{proposition}
\begin{proof}
Let $X \in \GTI$ be a cofibrant diagram and $U \colon \GTI \to \TI$ the forgetful functor. Recall that $\GTI$ is a cofibrantly generated model category where the generating cofibrations are $G/H \times S^{n-1} \hookrightarrow G/H \times D^n$ for all subgroups $H$ of $G$. The forgetful functor $U \colon \GTI \to \TI$, defined by  
    $$\big(G/H \times S^{n-1} \hookrightarrow G/H \times D^n\big) \mapsto   \big( \bigsqcup_{{|G/H|}} S^{n-1} \hookrightarrow  \bigsqcup_{{|G/H|}} D^n\big),$$
takes cofibrant diagrams to cofibrant diagrams. For a small category $\I$, any cofibrant diagram in $\TI$ is a closed filtration, see \cite[Proposition 5.1]{BL23}.
Thus, $U(X)$ is a closed filtration and so is $X$, since closeness and inclusion do not depend on the action of $G$.

\end{proof}

\begin{lemma}
    Given a pushout square
        $$\begin{tikzcd}
        A \arrow{r}{h}\arrow[swap]{d}{f} &X \arrow{d}{g} \\
        B \arrow[swap]{r}{s} & Y 
        \end{tikzcd}$$    
    in the category $\GTop$ with $f$ a closed inclusion, then $g$ is also a closed inclusion. 
\end{lemma}
\begin{proof}
   From the pushout diagram, we have
        $$Y= X\bigsqcup B/\sim,$$  
   where $f(a)\sim h(a)$ for all $ a\in A$. Since both $f$ and $h$ are $G$-maps, $Y$ inherits a $G$-action and $g$ becomes a $G$-map. We first verify that $g$ is an inclusion.
   Inclusion is straightforward for elements of $X$ that are not identified in the quotient.
   Now we consider $x_1 \neq x_2$ in $X$ so that $h(a)=x_1$ and $h(b)=x_2$ for some $a,b\in A$. Since $h(a)\sim f(a)$ and $h(b) \sim f(b)$ in $Y$,
   then $f(a)\ne f(b)$ in $B$ and so are their images in $Y$.
   Hence, $g$ is an inclusion.
   We now show that $g(X)=W$ is closed in $Y$. Note that $g^{-1}(W)$ is closed in $X$. In addition, $g^{-1}(W\cap B) =f(A)$ is closed in $B$ since $f$ is a closed inclusion. Thus, $g$ is a closed inclusion.
   
\end{proof}

Reall that, a \emph{directed set} is a nonempty \emph{preorder set} (reflexive and transitive, not necessarily antisymmetric) $\I$ such that for all $a,b\in \I$ there exists $c \in \I$ satisfying $a,b\le c$. Given a directed set \I, a functor $X\colon \I \to\GTop$ and $a\in \I$, let $\mu_a \colon X_a \longrightarrow \colim~ X$ denote the canonical $G$-map. 

    \begin{lemma}\label{lemma Closed inclusion}
        If {$\I$} is a directed set and $X\colon \I \to {\GTop}$ is a closed filtration. Then each map $\mu_a$ is a closed inclusion.
    \end{lemma}

    \begin{proof}
    This is a consequence of Corollary $2.23$ and \cite[Lemma 3.3]{SN09}.
    
    \end{proof}
    
The following notations are adapted from \cite{CSZ10} in our equivariant setting.
\begin{definition}
    For a directed set $\I$, a functor $X \colon \I \to\GTop$ is called \emph{$1$-critical} if $X$ is a closed filtration and for each $x\in \colim~X$, the set 
    $$S_x \coloneq \{ a\in \I \mid x\in \mbox{image} (\mu_a: X_a\to \colim X)\}$$ 
has a minimum element. 
\end{definition}

\begin{definition}\label{d-birth-index}
    For any $1$-critical diagram $X$ and $x \in \colim ~X$, the minimum element of $S_x$ is called \emph{birth index} of $x$. This defines a function (set map) $$\zeta^X \colon \colim~ X \to \I$$ sending each element $x \in \colim~X$ to its birth index.
\end{definition}

Note that the map $\zeta^X$ satisfies $\zeta^X(gx) = \zeta^X(x)$ since the birth index of $x$ and $gx$ should be same for any $g \in G$. In other words, the function $\zeta^X$ is equivariant with the trivial $G$-action on the index set $\I$. 
    \begin{proposition}\label{Prop_1 critical}
        For any directed set {$\I$}, each cofibrant diagram in {$\GTI$} is $1$-critical.
    \end{proposition}
    
\begin{proof}
    Let $X$ be a cofibrant diagram in {$\GTI$}. The forgetful functor $U \colon \GTI \to \TI$ on $X$ gives us a cofibrant diagram in $\Top$. Since any cofibrant diagram in $\Top$ is $1$-critical, see \cite[Proposition 5.4]{BL23}, $U(X)$ is $1$-critical.
    The colimit in $\GTI$ is the colimit in $\TI$ with a canonical group action, the birth index for the diagram in $\Top$ gives the birth index for the diagram in $\GTop$. Hence, $X$ is $1$-critical.
    
\end{proof}

Let $\I$ be a directed set. We define a category $\textbf{F}$ with 
    \begin{align*}
        \text{ob}(\textbf{F}) & \coloneq \{ \gamma^S \colon S \to \I \mid S \in \GTop ,~~ \gamma^S(gx) = \gamma^S(x)\}, \text{ and} \\ \text{hom}_{\textbf{F}}(\gamma^S, \gamma^T) & \coloneq \{f \in \hom_{\GTop}(S, T) \mid \gamma^T\circ f \le \gamma^S \}.
    \end{align*}
Let $\GTIC$ denote the full subcategory of $\GTI$ whose objects are the $1$-critical diagrams. 
From the functoriality of colimits, we can say that, for diagrams $X, Y\colon \I \to \GTop$, a natural transformation $f\colon X \to Y$ induces a $G$-map $\colim~ f \colon \colim~ X \to \colim~ Y.$
If $X,Y$ are $1$-critical, then $\zeta^Y \circ \colim~f \le \zeta^X$. Thus, we have a functor
    $$\fcolim \colon \GTIC \to \textbf{F}, \quad \quad \text{defined by } X \mapsto \zeta^X, $$ (see Definition \ref{d-birth-index}).
We have another functor 
    $$\mathscr{S} \colon\textbf{F} \to \GTI, \quad \quad \text{defined by } \mathscr{S}(\gamma^{S})_a\coloneq \{y\in S \mid \gamma^S(y)\le a \}.$$ 
This is a $G$-space since $\gamma^S$ has the property $\gamma^S(gx) = \gamma^S(x)$.
Clearly, $ \fcolim \circ \mathscr{S} = Id_{\textbf{F}}.$

\begin{proposition}\label{Prop_Id}
    For a directed set $\I$, we have a natural isomorphism ${\mathscr{S} \circ \fcolim \cong {Id}_{{{\GTI}}_{crit}}}$.
\end{proposition}

\begin{proof}
Let $X \in \GTIC$. First, we show 
    $$(\mathscr{S} \circ \fcolim~ X)_a = \text{im} ~ \mu_a,~~~a\in \text{\I}.$$
 Recall
    $$X \xrightarrow[]{\hspace{.3mm} \fcolim \hspace{.3mm}} \zeta^X \xrightarrow[]{\quad \mathscr{S} \quad}  \{({\zeta^X})^{-1}(-\infty,a]\}_{a\in I} \quad \text{ and } \quad X_a \lhook\joinrel\xrightarrow{\hspace{.4mm} \mu_a \hspace{.4mm}} \colim ~X.$$
Consider $x \in \text{im}~\mu_a$. Then $x \in X_a$, hence $\zeta^X(x)\leq a$ and $x\in({\zeta^X})^{-1}(-\infty,a]= (\mathscr{S} \circ \fcolim~ X)_a$. For the converse part, assume $x\in (\mathscr{S} \circ \fcolim ~X)_a$, then $x\in ({\zeta^X})^{-1}(-\infty,a]$ and $\zeta^X(x)\leq a$. Hence $x\in X_a$ and $x\in \text{im}~\mu^X_a$. Since $X$ is a closed filtration, by Lemma~\ref{lemma Closed inclusion}, $\mu^X_a$ is a closed inclusion, and hence, $\mu^X_a$ is a homeomorphism onto its image. 
For $a\le b \in \I$, we have $\mu^X_b\circ X_{a,b} = \mu^X_a$. Thus, these homeomorphisms define a natural isomorphism $\mu^X\colon X \to \mathscr{S} \circ \fcolim~X$. Assembling the natural isomorphisms $\{\mu^X\}_{X\in\GTIC}$ on $X$, we get a natural isomorphism 
    $$\text{Id}_{\GTIC} \xrightarrow{\hspace{.2mm} \cong \hspace{.2mm}} \mathscr{S} \circ \fcolim.$$       
\end{proof}  

The following Proposition is the key ingredient for our university result. It is an equivariant analogue of \cite[Proposition 5.7]{BL23}

\begin{proposition}\label{Stability of d_{GHI}}
    Let $X,Y \in$ {$\GTR$} are $(G,\delta)$-interleaved. Then there exist $T \in$ {$\GTop$} and equivariant functions $\gamma^X,\gamma^Y\colon T \to \RR$ (with trivial $G$ action on $\RR$) satisfying $$\mathscr{S} (\gamma^X)\simeq_G X,~ \mathscr{S} (\gamma^Y)\simeq_G Y, \quad\text{and}\quad d^G_\infty(\gamma^X,\gamma^Y)\le \delta.$$
\end{proposition}
\begin{proof}
    For a small category $\I$, let $\Q$ denote a cofibrant replacement functor in the projective model structure on $\GTI$.
    We treat the cases for $\delta=0$ and $\delta>0$ separately.
    
   Let $\delta=0$. Thus, we have an isomorphism $X \to Y$. Take $T \coloneq \colim~ \Q X$. Note that
   $\Q X$ is $1-$critical, by Proposition \ref{Prop_1 critical}. 
   Now let $$\gamma^X =\gamma^Y \coloneq \zeta^{\Q X}.$$
   Since $\mathbb{R}$ with its total order is a directed set, Proposition \ref{Prop_1 critical} and Proposition \ref{Prop_Id} imply that $\mathscr{S} (\gamma^X)= \mathscr{S} (\gamma^Y) \cong \Q X$. 
   Using the weak equivalence $\Q X \to X$ and the isomorphism $X\to Y$, we get our desired result for $\delta=0$ as 
        $$\mathscr{S} (\gamma^X)\cong \Q X \simeq X \quad \text{and} \quad \mathscr{S} (\gamma^Y) \cong \Q X \simeq X \cong Y.$$
   
   Now let $\delta>0$. Recall the definition of the interleaving category $\mathcal{I}^{\delta}$ and the functors $E^0, E^1\colon\mathbb{R} \to \mathcal{I}^{\delta}$ from Subsection \ref{Subsec_Homotopy Interleaving}.
   Note that $\mathcal{I}^{\delta}$ is a poset category and the underlying poset is a directed set for $\delta>0$. Since $X$ and $Y$ are $\delta$-interleaved, there exists a functor $Z \colon \mathcal{I}^{\delta} \to \GTop$ such that $Z\circ E^0 = X$ and $Z\circ E^1= Y$.
   Now we define $T \coloneq \colim~ \Q Z$ and show $T$ satisfies our desired property.
   
   Note that $E^0$ and $E^1$ both are final functors, see \cite[Section 8.3]{ER14}. Hence, we have a canonical identifications of $\colim~(\Q Z \circ E^0)$ and $\colim~(\Q Z \circ E^1)$ with $T$ such that for each $r\in \mathbb{R}$,
        $$\mu^{\Q Z}_{(r,0)} = \mu^{\Q Z \circ E^0} _r \quad \text{and} \quad \mu^{\Q Z}_{(r,1)} = \mu^{\Q Z \circ E^1} _r.$$
    We claim that $\Q Z \circ E^0$ and $\Q Z \circ E^1$ are both $1$-critical.
    Without loss of generality, we show this for $\Q Z \circ E^0$.
    By Proposition \ref{Prop_1 critical}, $\Q Z $ is $1$-critical.
    For each $a \in \mathcal{I}^\delta$, $\mu ^{\Q Z}_a\colon (\Q Z)_a \to T $ is a closed inclusion, by Proposition \ref{lemma Closed inclusion}.
    Therefore, for each $r \in \mathbb{R}$, $\mu^{\Q Z \circ E^0}_r \colon (\Q Z \circ E^0)_r \to T$ is a closed inclusion.
    Since $\Q Z$ is $1$-critical, for each $z\in T$, there exists a minimum element $(r,j) \in \text{ob}~ \mathcal{I}^\delta$ such that $z \in \text{im}~\mu^{\Q Z}_{(r,j)} $.
    We then have 
        $$r+j\delta = \text{min}~ \{s\in \mathbb{R} \mid z\in \text{im}~ \mu^{\Q Z \circ E^0}_s \}.$$
    Thus $\Q Z \circ E ^0$ is $1$-critical.
    Similarly, $\Q Z \circ E ^1$ is $1$-critical.
    We now define $\gamma^X , \gamma^Y \colon T \to \mathbb{R}$ by
        $$ \gamma^X = \zeta^{\Q Z \circ E^0} \quad \text{and} \quad \gamma^Y = \zeta^{\Q Z \circ E^1}.$$
     By Proposition \ref{Prop_Id}, $\mathscr{S} (\gamma^X)\cong \Q Z \circ E^0$, $\mathscr{S} (\gamma^Y)\cong \Q Z \circ E^1.$
     Also, we have a weak equivalence $\Q Z \to Z$ by the construction. A restriction of that provides us with the weak equivalences $\Q Z \circ E^0 \to X$ and $\Q Z \circ E^1 \to Y.$
     Hence, $$\mathscr{S} (\gamma^X)\cong_G \Q Z \circ E^0 \simeq_G X \quad \text{and} \quad \mathscr{S} (\gamma^Y)\cong_G \Q Z \circ E^1 \simeq_G Y.$$

     Now we show $d^G_\infty(\gamma^X,\gamma^Y)\le \delta$ to complete the proof.
     Let $z \in \colim~ \Q Z$. Then there is a minimum index $(r,j) \in \text{ob}~ \mathcal{I}^ \delta $ such that $z \in \text{im}~ \mu^{\Q Z}_{(r,j)}$.
     if $j=0$ then $\gamma^X(z)=r \text{ and } \gamma^Y(z)=r+\delta$. On the other hand, if $j=1$ then $\gamma^Y(z)=r \text{ and } \gamma^X(z)=r+\delta$.
     In either case, we have $|\gamma^X(z) - \gamma^Y(z)| = \delta.$
     This holds for all $z \in T$. We have $d^G_\infty(\gamma^X,\gamma^Y)\le \delta$.
\end{proof}

We conclude this section with the universality of the $G$-homotopy interleaving distance $d^G_{\HI}$.

\begin{theorem}\label{Th_Universality}
    Let $d^G$ be a $G$-stable and $G$-homotopy invariant  on $\GTR$. Then $d^G \leq d^G_{\HI}$.
\end{theorem}

\begin{proof}
  Let $X,Y \in \GTop^\RR$ with $d^G_{\HI}=\epsilon=\inf \{ \delta \colon X \text{ and } Y \text{ are } (G, \delta) \text{-homotopy interleaved}\}$. Then $X,Y$ are $(G,\epsilon')$-homotopy interleaved for any $\epsilon'>\epsilon$. Hence, there exists $X',Y'\in \GTop$ such that $X'\simeq_G X$, $Y'\simeq_G Y$, and $X'$ and $Y'$ are ($G, \epsilon'$)-interleaved. Then by Proposition \ref{Stability of d_{GHI}} there exists a $G$-topological space $T$ and two $G$- maps $\gamma^{X'},\gamma^{Y'} \colon T \to \RR$ such that $\mathscr{S}(\gamma^{X'}) \simeq_G X'$ and $\mathscr{S}(\gamma^{Y'}) \simeq_G Y'$, and $d^G_\infty(\gamma^{X'},\gamma^{Y'}) \leq \epsilon'$. Consider $d^G$ be a any $G$-stable and $G$-homotopy invariant distance on $\GTop^\RR$. Hence, by the stability of $d^G$, we have $d^G(\mathscr{S}(\gamma^{X'}),\mathscr{S}(\gamma^{Y'})) \leq \epsilon'$. Using stability and triangle inequality of $d^G$ we have $d^G(X,Y) \leq \epsilon'$. Since $\epsilon'$ was arbitrarily, we have $d^G(X,Y)\leq \epsilon =d_{\HI}^G(X,Y)$. Therefore, the distance $d_{\HI}^G$ is univarsal.
  
\end{proof}
\end{mysubsection}


\section{Equivariant Persistent Whitehead Theorem}
In this section, we investigate the equivariant analogue of the persistent Whitehead theorem. 
The classical Whitehead theorem says that if a map $f$ between CW-complexes induces isomorphisms between the $n$-th homotopy groups for all $n$, then $f$ is a homotopy equivalence. Since the homotopy interleaving distance is a homotopy-theoretic notion of approximate weak equivalence for filtrations, it motivates the study of persistent version of the Whitehead theorem. 

First we need the persistent analogues of homotopy groups and homotopy equivalence. For a diagram of based spaces $X \in \Top_{*}^{\RR}$ and $n \geq 0$, the $n$-th based persistent homotopy group of $X$ is defined by the composite functor $\pi_nX$, see \cite[Definition 8.1]{BL23}.
For $\delta \geq 0$, two $\RR$-spaces $X$ and $Y$, a pair of morphisms $f \colon X \to Y(\delta)$ and $g \colon Y \to X(\delta)$ are $\delta$-\emph{homotopy equivalences} if 
    $$g(\delta) \circ f \simeq \Phi^{X, 2 \delta}  \quad \text{and} \quad f(\delta) \circ g \simeq \Phi^{Y, 2 \delta}$$
where $\Phi^{X,\delta} \colon X \to X(\delta)$ is the morphism that assemble the structure maps, $f(\delta) \colon X(\delta) \to Y(2\delta)$ is induced by the map $f$; and $g(\delta)$, $\Phi^{Y,\delta}$ are defined analogously, see \cite[Definition 8.4]{BL23}. For two $\RR$-spaces $X$ and $Y$, if such $f$ and $g$ exist for some $\delta \geq 0$, we call $X$ and $Y$ are $\delta$-\emph{homotopy equivalent}.
This leads to the natural query: for connected, cofibrant $X,Y \in \Top^{\RR}$, any $\delta \geq 0$, if a morphism $f \colon X \to Y(\delta)$ induces $\delta$-interleaving on all persistent homotopy groups, 
    \begin{enumerate}
        \item is $f$ a $\delta$-homotopy equivalence?

        \item are $X$ and $Y$ $\delta$-homotopy interleaved (see Definition \ref{def_hinterleaved})?
    \end{enumerate}
\noindent After producing counter example (\cite[Example 8.6]{BL23}) to such naive guesses, Blumberg--Lesnick conjectured a persistent analogue of the classical Whitehead theorem \cite[Conjecture 8.7]{BL23}.

A direct approach to prove this conjecture using the homotopy interleaving distance is technically challenging, since homotopy interleavings are defined through homotopy coherent diagrams rather than morphisms in the homotopy category. 
To overcome this difficulty, Lanari and Scoccola \cite{LS23} considered the interleaving distance in the homotopy category to establish a version of the persistent Whitehead theorem. To investigate the equivariant setting, we introduce $(G,\delta)$-interleaving distance in homotopy category in this article.
Here we want to remark that our proof of equivariant persistent Whitehead theorem follows similar techniques used by Jardine\cite{Jar20} and Lanari-Scoccola\cite{LS23} for each fixed point set, thereby obtaining the required lift in the equivariant setting without changing the right shifting parameter, as in the non-equivariant case.

Although we are mostly interested in $G$-spaces, most of the results can also be obtained for $G$-simplicial sets since the equivariant realization--singular adjunction
    $$|-|:G\text{-sSet}\rightleftarrows G\text{-Top}:\operatorname{Sing}$$
is a Quillen equivalence, i.e., $G$-simplicial sets and $G$-spaces present equivalent equivariant homotopy theories. Throughout this section, we use the unified notion $G$-$\mathbb{S}$ for both $\GTop$ and $\GsS$. In particular, $\mathbb{S}$ will mean either $\Top$ or $\sSet$.

\begin{definition}[$G$-interleaving in homotopy category] \label{Defn_GIHC}
    Let $X, Y \in \GSR$. We call $X, Y$ are $(G,\delta)$-\emph{interleaved in homotopy category} if they are $\delta$-interleaved in $\Ho(\GSR)$. The $(G,\delta)$-\emph{interleaving distance} between $X$ and $Y$ in homotopy category is defined by 
        $$d_{\IHC}^G (X,Y) \coloneq \inf \{ \delta \colon X \text{ and } Y \text{ are } (G, \delta) \text{-interleaved in homotopy category}\}.$$
\end{definition}

In the classical Whitehead theorem, we need connectedness of a CW-complex. To establish an equivariant analogue of persistent Whitehead theorem, we first define $G$-connectedness of an equivariant CW-complex. 

\begin{definition}
    A $G$-CW complex $T$ is called $G$-\emph{connected CW-complex} if $T^H$ is non-empty and connected for all subgroups $H$ of $G$ where $T^H$ is the fixed-point set by $H$. We call a diagram $X \in \GTop^{\RR}$ is $G$-\emph{connected diagram} if $X_t$ is $G$-connected CW complex for all $t \in \RR$.
\end{definition}

Throughout this section, by a $G$-CW-complex, we mean a $G$-connected CW-complex unless otherwise mentioned.

The equivariant persistent homotopy group is motivated by the non-equivariant version by Jardine\cite{Jar20}. Recall that for a pointed $G$-space $T$ and a subgroup $H$ of $G$, the $n$-\emph{th} $H$-\emph{equivariant homotopy group} of $T$ is the ordinary homotopy group of the fixed point set $T^H$, i.e., 
    $$\pi_n^H(T) \coloneq \pi_n(T^H).$$ 
The collection $\{\pi^H_n(T) \colon H \leq G\}$ is called the \emph{equivariant homotopy group} of $T$. A map $f \colon T \to T'$ of $G$-spaces is a \emph{weak equivalence} if for all $n$ and all subgroups $H$ of $G$, the induced maps $f_* \colon \pi_n(T^H) \to \pi_n(T'^H)$ are isomorphisms.

Let $X \in G$-$\mathbb{S}^{\RR}$. For any subgroup $H \leq G$ the $H$-\emph{equivariant persistent set} $\pi_0^H(X) \colon \RR \to \mbox{Set}$ is defined by $\pi_0^H \circ X$. Also, for $n \geq 1$, $t \in \RR$ and $x \in X_t$; the $n$-\emph{th} $H$-\emph{equivariant persistent homotopy group} of $X$ based at $x$ is the persistent group $\pi_n^H (X,x) \colon \RR \to \mbox{Grp}$ which is defined by 
    $$\pi_n^H (X,x)=
    \begin{cases}
        \{e\} &\text{when } s<t,\\
        \pi_{n}(X_s^H, X^H_{t,s}(x)) \quad &\text{when } s\geq t;
    \end{cases}$$
where $\{e\}$ is the trivial group and $X^H_{t,s}: X^H_t \to X^H_s$ are structure maps. Now we define equivariant $\delta$-interleaving in homotopy groups.

\begin{definition}
    Let $G$ be a finite group, $\delta$ be a non-negative real number and $f \colon X \to_{\delta} Y$ a map in $G$-$\mathbb{S}^{\RR}$.
    We say $f$ induces a \emph{$(G,\delta)$-interleaving in homotopy groups} if for all subgroups $H$ of $G$ the following are satisfied
        \begin{enumerate}[(i)]
            \item the induced map $\pi_0^H(f) \colon \pi_0^H(X) \to_{\delta} \pi_0^H(Y)$ is a part of a $\delta$-interleaving of persistent sets, 
            \item if for every $r \in \RR$, every $x\in X(r)$, and every $n\geq 1 \in \mathbb{N}$, the induced map 
                $$\pi_n^H(f)\colon \pi_n^H(X,x) \to_\delta \pi_n^H(Y,f(x))$$ 
            is part of a $\delta$-interleaving of persistence groups.
        \end{enumerate}
\end{definition}

A fibration of Kan complexes that induces an isomorphism in all homotopy groups has the right lifting property with respect to cofibrations, see \cite[Theorem I.7.10]{GJ99}. Following the ideas used in Jardine \cite{Jar20}, a persistent version has been proved by Lanari--Scoccola in \cite[Corollary 5.13]{LS23} where they proved for a fibration of fibrant objects (in $\mathbb{S}^{\RR}$) inducing a $\delta$-interleaving in homotopy groups, the lift exists up to a shift. 
Here we prove an analogous equivariant version of this result in Corollary \ref{cor:lift}. For this we need the following definitions.

\begin{definition}
    Let $X, Y\in G$-$\mathbb{S}^{\RR}$ and $n\in\mathbb{N}$. A map $h\colon X \to Y$ is an \emph{equivariant} $n$-\emph{dimensional extension} if there exist real numbers $\{r_i\in \RR\}_{i\in I}$ and subgroups of $\{H_j\mid {j\in J}\}$ and commutative squares given left below, that give rise to the pushout square on the right
    $$
        \xymatrix{
        G/H_j \times \partial D^n \ar[rr]^{\quad f_i} \ar@{^{(}->} [d]_{} & & X(r_i) \ar[d]^{h_{r_i}} & & \bigsqcup\limits_{i \in I} r_i \odot (\bigsqcup\limits_{j \in J}G/H_j \times \partial D^n) \ar[rr]^{\hspace{1.5cm} f} \ar@{^{(}->} [d]_{} & & X \ar[d]^{h}\\
        G/H_j \times D^n \ar[rr]_{\quad g_i} & & Y(r_i) & & \bigsqcup\limits_{i \in I} r_i \odot (\bigsqcup\limits_{j \in J}G/H_j \times D^n) \ar[rr]_{\hspace{1.5cm} g} & & Y.}
    $$
(See Subsection \ref{Diagram Category} for the notation $r\odot -$). Here the action on both $\partial D^n$ and $D^n$ is assumed to be trivial. When $G$-$\mathbb{S}=\GTop$, $\partial D^n$ and $D^n$ denote the sphere $S^{n-1}$ and the disk $D^n$, respectively. When $G$-$\mathbb{S}=\GsS$, they denote the boundary $\partial\Delta^n$ and the standard simplex $\Delta^n$, respectively. One can visualize that $Y$ is obtained from $X$ by attaching equivariant $n$-cells at filtration level $\{r_i\colon i\in I\}.$
\end{definition}

\begin{definition}
    let $\iota \colon X\to Y$ be a projective cofibration in $G$-$\mathbb{S}^{\RR}$ and let $n\in\mathbb{N}$. We say $\iota$ is an \emph{equivariant} $n$-\emph{cofibration} if it factors as a composite of $n+1$ maps $f_0, f_1, \dots ,f_n$ with each $f_i$ an equivariant {$n_i$}-dimensional extension for some $n_i\in \mathbb{N}$. We say $X\in G$-$\mathbb{S}^{\RR}$ is \emph{equivariant} $n$-\emph{cofibrant} if the map $\emptyset \to X$ is an equivariant $n$-cofibration.  
\end{definition}

Now we investigate when a diagram on $G$-simplicial set or a $G$-space is equivariant $n$-cofibrant, see Proposition \ref{Prop_G-Sset-n-cofibrant} and Proposition \ref{Prop_G-CW-n-cofibrant}. To do that we need the notions of filtered $G$-simplicial set and persistent $G$-CW-complex. 

\begin{definition}
    Let $P$ be a poset $(P, \leq)$. Consider it with the trivial $G$-action. A \emph{$P$-filtered $G$-simplicial set} $(X, \beta)$ is a $G$-simplicial set $X$ equipped with $G$-maps $\beta_{n} \colon X^n \to P$, satisfying the following
        \begin{enumerate}[(i)]
            \item $\beta_{n-1}(d_i(\sigma)) \leq \beta_n (\sigma) ~\forall n \geq 1, \sigma \in X^n$ and a boundary map $d_i \colon X^n \to X^{n-1}$,
            
            \item $\beta_{n+1}(s_i(\sigma)) \leq \beta_n(\sigma) ~ \forall n \geq 0, \sigma \in X^n$ and a degeneracy map $s_i \colon X^n \to X^{n+1}$.
        \end{enumerate}
    
\end{definition}

Given a $P$-filtered $G$-simplicial set $(X, \beta)$, we get a persistent $G$-simplicial set $\widehat{({X}, \beta)} \in \GsS^P$ defined by $\widehat{({X} ,\beta)} (p)^n= \{ \sigma \in X^n \colon \beta_n(\sigma) \leq p\}$ for $p \in P$. The faces and degeneracies are given by restricting the ones of $X$.
We call a persistent $G$-simplicial set $Y\in \GsS^P$ \emph{filtered} if it is isomorphic to $\widehat{({X} ,\beta)}$ for some $(X,\beta)$. 
The following result gives a useful criterion for determining when a persistent $G$-simplicial set is filtered.

\begin{lemma}\label{Filterted iff condition}
    A persistent $G$-simplicial set $X\in {\GsS}^P$ is a filtered $G$-simplicial set if and only if the following conditions are satisfied:
    \begin{enumerate}[(i)]
        \item For any $r \leq r'$, the equivariant structure map $X_r \to X_{r'}$ is a monomorphism. In particular, we assume $X_r$ is a subsimplicial set of $X_{r'}$, up to an isomorphism.

        \item For any $r\in P$ and $\sigma \in X_r\in \GsS$, the set $\{t \in P\mid \sigma \in X_t\}$ has a minimum.
    \end{enumerate}
\end{lemma}

\begin{proof}
   Let $X\in \GsS^{P}$ be filtered, and let $\sigma\in X_r$ be an $n$-dimensional $G$-simplex. Then $\beta_n(\sigma)\le r$. and also $\beta_n(g.\sigma)\le r$ for all $g\in G$. Consequently, for every $r'\ge r$, we also have $\beta_n(\sigma)\le r'$ and $\beta_n(g.\sigma)\le r'$ for all $g\in G$, so the structure map $X_r\to X_{r'}$ is equivariant and sends $g.\sigma$ to itself. This verifies (i). Moreover, since $X$ is filtered, the value $\beta_n(\sigma)$ is, by definition, the least parameter at which $\sigma$ appears. Thus, (ii) follows.
   
   Conversely, let $X\in \GsS^{P}$ satisfies (i) and (ii), we will show it is filtered. Since each $X_r$ is a $\GsS$, Condition (i) implies $Z=\bigcup_{r\in P} X_r$ is well defined simplicial set and inherits a well defined simplicial $G$-action. For $\sigma \in Z^n$ define $\beta_n(\sigma)\coloneq \text{min}\{r\in P \mid  \sigma \in X_r\}$, which is well-defined from (ii). It now remains to prove $\beta_n$ is a $G$-map. Observe that $\sigma \in X_t$ if and only if $g.\sigma\in X_t$ for all $g\in G$. Hence, $\{r\in P \mid  \sigma \in X_r\}=\{r\in P \mid  g.\sigma \in X_r\}$. Taking the minimum on both sides, we get $\beta_n(g.\sigma)=\beta_n(\sigma)$ for all $g\in G$. Thus, $\beta$ is $G$-invariant. Hence, $X$ is isomorphic to the filtered object $\widehat{(Z, \beta)}$.
   
\end{proof}

\begin{lemma}\label{cell attachment remains filtered}
    \begin{enumerate}
        \item A retract of a filtered $G$-simplicial set is filtered.

        \item If the domain of an equivariant cell attachment is a filtered $G$-simplicial set, the codomain is also a filtered $G$-simplicial set.

        \item Let $\zeta$ be a limit ordinal and let $Z_{\bullet} \colon \zeta \to \GsS^P$ be a diagram of persistent $G$-simplicial sets, where for each $\alpha < \zeta$ we have that the map $Z_{\alpha} \to Z_{\alpha+1}$ is an equivariant cell attachment.
        If $Z_{\alpha}$ is a filtered simplicial set for every $\alpha < \zeta$, then $Z_{\zeta}=\colim_{\alpha < \zeta} Z_{\alpha}$ is a $G$-filtered simplicial set.
    \end{enumerate}
\end{lemma}

\begin{proof}
     $(1)\colon$ Let $X$ be a filtered $G$-simplicial set. $A$ is a retract of $X$. There exists $G$-equivariant morphisms $\kappa\colon X\to A$ and $\iota \colon A\to X$ such that $\kappa\circ \iota=\text{id}_A$. We verify conditions (i) and (ii) of Lemma \ref {Filterted iff condition} to prove $A$ is filtered. 
    Consider the following diagram
        $$
            \xymatrix{
            A_r \ar[rr]^{\iota_r} \ar[d]_{A_{r,s}} & & X_r \ar[d]^{X_{r,s}}\\
            A_s \ar[rr]_{\iota_s} & & X_s , 
            }
        $$
    where all $\iota_s,\iota_r$, and $X_{r,s}$ are monomorphisms and the diagram commutes. Thus, $A_{r,s}$ is a monomorphism for all $r \to s$ in $P$. Now, for condition (ii), consider $\sigma\in \bigcup_{r\in P} A_r$. Then $\iota(\sigma)\in \bigcup_{r\in P}X_r$. Since $X$ is filtered, the set 
        $S_{\iota(\sigma)}=\{t\in P\mid \iota(\sigma)\in X(t)\}$
    has a minimum, say $t_0$. Hence, $\iota(\sigma) \in X_{t_0}$. Since, $\kappa_{t_0} \colon X_{t_0} \to A_{t_0}$ is a retraction, $\kappa_{t_0}(\iota(\sigma))=\sigma \in A_{t_0}$. We have,
        $$ \{t\in P\mid \sigma\in A_t \} \subseteq \{t\in P\mid \iota(\sigma)\in X_t\}=S_{\iota(\sigma)}$$
    Thus, for every $G$-simplex $\sigma$ the set $\{t\in P\mid \sigma\in A_t\}$ has a minimum. Therefore, $A$ is filtered.

    $(2)\colon$ Let $X$ be a filtered $G$-simplicial set. Attaching an  equivariant cell to $X$ corresponding to a subgroup $H\leq G$ gives rise to the pushout square below:
        $$\xymatrix{
        r \odot (G/H \times \partial \Delta^n) \ar[rr]^{\hspace{.8cm}\phi_G} \ar@{^{(}->} [d]_{} & & X \ar[d]^{} & & \\
        r \odot (G/H \times \Delta^n) \ar[rr]_{} & & Y & & 
        }$$ 
     The left vertical inclusion is a level-wise $G$-monomorphism, and $\phi_G$ is a $G$-map. Hence, $Y$ inherits a canonical $G$-action. We know that equivariant monomorphisms are preserved under pushout in the category of $\GsS$. Thus, the induced map $X\to Y$ is also an equivariant monomorphism level-wise in $\GsS^P$. Since $X$ is filtered, the structural maps $X_r\to X_s$ are equivariant monomorphisms. This leads us to the commutative diagram below
        $$\xymatrix{
            X_r \ar[rr]^{a_r} \ar[d]_{X_{r,s}} & & Y_r \ar[d]^{Y_{r,s}}\\
             X_s \ar[rr]_{a_s} & & Y_s}$$
    where $a\colon X \to Y$ comes from the pushout are equivariant monomorphism. Therefore, each $Y_{r,s}$ is an equivariant monomorphism. 

    Consider $\sigma \in \bigcup_{r\in P} Y_r$. Then two cases may arise. First, let $\sigma\in X$ before the attachment, or it is a boundary simplex of the new cell that is glued to $X$ via $\phi_G$. Since $X$ is already filtered and the map $X\to Y$ is monomorphism, $\text{min}\{t\in P \mid \sigma \in X_t\}=\text{min}\{t\in P \mid \sigma \in Y_t\}$. 
    In the other case, let $\sigma$ be the nondegenerate  $n$-dimensional $G$-simplex of the new cell or a boundary simplex which is part of the interior of the new cell, or any degeneracy of such simplices, which arises from the equivariant cell attachment $s \odot (G/H \times \partial \Delta^n) \hookrightarrow s \odot (G/H \times \Delta^n)$. 
    In such cases, $s=\text{min}\{t\in P \mid \sigma \in Y_t\}$. So, in any case, the set $\{t\in P \mid \sigma \in Y_t\}$ has a minimum. Therefore, $Y$ is filtered.

    $(3)\colon$ Let $\zeta$ be a limit ordinal and let $X_\bullet \colon \zeta \to \GsS^P$ be a diagram of persistent $G$-simplicial sets. We show $X_\zeta = \text{colim}_{\gamma<\zeta}X_\gamma$ is filtered. Since each $X_\gamma$ is a $G$-simplicial set, the colimit $X_\zeta$ inherits a canonical $G$-action. Hence, $X_\zeta$ is a $G$-simplicial set and $X_\zeta(s)\to X_\zeta(t)$ is a $G$-map for all $s<t$. Now, we will show that $f_{st}\colon X_\zeta(s) \to X_\zeta(t)$ is an equivariant monomorphism. Let $a,b\in X_\zeta(s)$ such that $f_{st}(a) = f_{st}(b)$. It is given that each map $X_\gamma(s) \to X_{\gamma+1}(s)$ is an equivariant cell attachment, hence an equivariant monomorphism by $(2)$. Hence, $X_\zeta(s)= \bigcup_{\gamma< \zeta}X_\gamma(s)$. We have,  $a,b \in X_\zeta(s)= \bigcup_{\gamma< \zeta}X_\gamma(s)$. So, there exists $\gamma_1,\gamma_2$ such that $a\in X_{\gamma_1}(s)$ and $b\in X_{\gamma_2}(s)$. let $\gamma' = \text{max}\{\gamma_1, \gamma_2\}$. Hence, $a,b \in X_{\gamma'}(s)$. We also have $f_{st}(a) = f_{st}(b)\in X_\zeta(t)= \bigcup_{\gamma< \zeta}X_\gamma(t)$. So there exists $\gamma''$ such that $f_{st}(a) = f_{st}(b)\in X_{\gamma''}(t)$. Take $\gamma = \text{max}\{\gamma', \gamma''\}$. Hence, $a,b \in X_{\gamma}(s)$ and $f_{st}(a) = f_{st}(b)\in X_{\gamma}(t)$. Since $X\gamma$ is filtered, the map $X_\gamma(s)\to X_\gamma(t)$ is an equivariant monomorphism. Thus, we have $a=b$. Therefore, $f$ is an equivariant monomorphism.

    We now prove for an $G$-simplex $\sigma \in \bigcup_{r\in P} X_\zeta(r)$ the set $\{t\in P \mid \sigma \in X_\zeta(t)\}$ has a minimum. let $\sigma \in X_\zeta(s)$ for some $s\in P$. Now $X_\zeta(s)=\bigcup_{\gamma<\zeta}X\gamma(s)$. Hence, $\sigma \in X\gamma(s)$. Since $X_\gamma$ is filtered the set $\{t\in P \mid \sigma \in X_\gamma (t)\}$ is non-empty and has a minimum, say $t_0$. So, $\sigma$ first apper at $X_\gamma(t_0)$. Since, $X_\gamma(t_0)\to X_{\gamma+1}(t_0)$ is an equivariant monomorphism $\sigma$ will first appear in $X_\zeta(t_0)$. We also have $X\zeta(s) \to X_\zeta(t)$ is an equivariant monomorphism. Hence $t_0=\text{min}\{t\in P \mid \sigma \in X_\zeta(t)\}$. Therefore, $X_\zeta$ is filtered.
    
\end{proof}

The following proposition provides a recognition criterion for projective cofibrant persistent $G$-simplicial sets.

\begin{proposition}\label{Prop_filtered iff projective cofibrant}
    A persistent $G$-simplicial set in $\GsS^{\RR}$ is filtered iff it is cofibrant in the projective model structure on $\GsS^{\RR}$.
\end{proposition}

\begin{proof}
    We know that the cofibrant objects in a cofibrantly generated model category are precisely the retracts of transfinite composition of cell attachments, see\cite[Proposition 2.1.18(b)]{Hovey}. Let $X$ be projective cofibrant. Now, an empty diagram is always filtered. If we attach an $G$-cell, it's still filtered by Lemma \ref{cell attachment remains filtered} $(2)$. Then the transfinite composition of cell attachment and its retract are still filtered by \ref{cell attachment remains filtered}$(3)$ and $(1)$, respectively. Hence, $X$ is filtered. Conversely, if $X$ is filtered, we can build it from an empty diagram by a transfinite composition of equivariant cell attachments. Hence, $X$ is projective cofibrant. 
    
\end{proof}

\begin{definition}
A persistent $G$-simplicial set $X$ is said to be \emph{pointwise $n$-skeletal} if, for every index $i$, the $G$-simplicial set $X_i$ contains cells only up to dimension $n$. 
\end{definition}
The following result provides a class of examples of $n$-cofibrant persistent $G$-simplicial sets. It is an immediate consequence of Proposition \ref{Prop_filtered iff projective cofibrant}.

\begin{proposition}\label{Prop_G-Sset-n-cofibrant}
    Let $n\in \mathbb{N}$. If $X\in {\GsS}^\RR$ is cofibrant in the projective model structure and pointwise $n$-skeletal, then it is equivariant $n$-cofibrant.
\end{proposition}

\begin{proof}
Since $X$ is a projective cofibrant, by Proposition \ref{Prop_filtered iff projective cofibrant}, $X$ is filtered. Hence, every $n$ simplex $\sigma$ has a value $\beta_n(\sigma)$. Recall that each simplex $\sigma \in \colim~ X$ has a minimum birth index. We start with an empty diagram. Then attach all $0$-dimension cells with respect to their birth indices. This is an equivariant $0$-dimensional extension which is denoted by $f_0$. 
Next, a similar attachment of $1$-cells respecting their birth indices gives us equivariant $1$-dimensional extension $f_1$. Continuing in a similar manner, the process ends at $n+1$-th step $f_n$ since the simplex is of dimension $n$. Hence $\sigma$ is equivariant $n$-cofibrant.

\end{proof}

\begin{example}\label{Eg_n cofibrant}
    Let $M$ be a metric space with an isometric $G$ action with exactly $n+1$ orbits. Then the Vietoris--Rips complex $\VR(M)$ inherits a simplicial $G$-action and it is equivariant $n$-cofibrant. 
\end{example}

In a simplicial set, if a $2$-cell is attached to the $0$-skeleton, its $1$-dimensional faces are necessarily present as part of the simplicial structure. In contrast, this need not hold for topological spaces: a $2$-cell can be attached directly to the $0$-skeleton, without any $1$-cells being attached first. Thus, cells in a topological space need not be attached in increasing order of dimension. Consequently, an analogue of Proposition \ref{Prop_G-Sset-n-cofibrant} cannot be expected to hold automatically in $G$-topological spaces. This motivates the following definition.

\begin{definition}
    Let $n\in \mathbb{N}$. A persistent $G$-topological space $X\in \GTI$ is an \emph{$n$-dimensional persistent $G$-CW-complex} if the map $\emptyset \to X$ can be factored as a composite of maps $f_0, \dots f_n$ where each $f_i$ is an equivariant $i$-dimensional extension for $i=0, \dots, n$.
\end{definition}

\begin{proposition}\label{Prop_G-CW-n-cofibrant}
    Every $n$-dimensional persistent $G$-CW-complex is equivariant $n$-cofibrant.
\end{proposition}

\begin{proof}
   Let $T$ be an $n$-dimensional persistent $G$-CW-complex. So it comes from a filtration. Since $T$ is $n$-dimensional, it has $n$-dimensional $G$-cells of dimension $0,1, \dots, n$ and it is built by attaching those cells. Thus $T$ is an equivariant $n$-cofibration. 
   
\end{proof}

\begin{example}
    The geometric realization of $\VR(M)$ which is defined in Example \ref{Eg_n cofibrant} is equivariant $n$-cofibrant.
\end{example}

As we said earlier, the lifting in the equivariant persistent case exists up to a shift, in particular, ``a right shift". In the following, we make sense of the notion right shift. Recall the shift functor from \eqref{Eq_shift functor}.

\begin{definition}
    Let $\delta$ be a non-negative real number, and $i \colon A \to B$ and $p \colon X \to Y$ be morphisms in $G$-$\mathbb{S}^{\RR}$. Let there be $G$-maps $A \to X$ and $B \to Y$ such that the square on the left below commute. If for all such commutative diagrams and any subgroups $H \leq G$, there exist a diagonal map $B^H \to X(\delta)^H$ as in the right below diagram such that the diagram is commutative, we say $p$ has the \emph{right} $(G,\delta)$-\emph{lifting property} with respect to $i$.
        $$
        \xymatrix{
        A\ar[rr] \ar[d]_{i} && X\ar[d]^{p}  && A^H \ar[rr] \ar[d]_{i^H} & & X^H \ar[d]\ar[rr]^{} \ar[d]^{p^H} & & X(\delta)^H \ar[d]^{p(\delta)^H}\\
         B\ar[rr] && Y && B^H \ar[rr]_{} \ar@/^.5pc/@{-->}[urrrr] & & Y^H \ar[rr]_{} & & Y(\delta)^H.
        }
        $$   
\end{definition}

We now prove the main results of this section, Corollary \ref{cor:lift} and Theorem \ref{Th_interleaving in homotopy category}. The proof relies on the following lemma. Jardine \cite{Jar20} showed that fibrations inducing interleavings on homotopy groups satisfy a shifted right lifting property. 
The corresponding non-equivariant result was established by Goerss--Jardine in \cite[Theorem I.7.10]{GJ99}. In the equivariant setting, we apply analogous arguments on the fixed-point sets $X^H$, for each subgroup $H\leq G$, to obtain the required lifts. Since the relevant $G$-orbits are disjoint, the lifts on the individual fixed-point sets can be assembled to produce a global equivariant lift. The following lemma makes this argument precise. Recall from \cite{MS10} that a $G$-simplicial set $\K$ is said to be a $G$-Kan complex if $\K^H$ is a Kan complex for each subgroup $H$ of $G$.

\begin{lemma}
Let $p\colon \mathcal{K} \to \mathcal{L}$ be a $G$-fibration between two $G$-Kan Complexes $\K$ and $\cL$. This means $p\colon \K^H \to \cL^H$ is a Kan fibration for all subgroups $H$ of $G$. Now, for a subgroup $H$ of $G$, we consider the following commutative square of $G$-simplicial sets 
    \begin{equation}\label{Eq_lift in Kan}
        \xymatrix{
        G/H \times \partial \Delta^n \ar[rr]^{\alpha} \ar@{^{(}->} [d]_{} & & \K \ar[d]^{p}\\
        G/H \times \Delta^n \ar[rr]_{\beta} & & \cL  
        }
    \end{equation}
where $\alpha$ and $\beta$ are $G$-maps. Now if we replace $\alpha$ and $\beta$ by their homotopic maps as in the diagram left in \eqref{Eq_lift in Kan 2}, it becomes the diagram on the right in \eqref{Eq_lift in Kan 2}. If the lifting problem on the right has a solution, then so does the initial square in \eqref{Eq_lift in Kan}.
Moreover, if the lift exists corresponding to two different subgroups $ H$ and $ K$, then we combine them to get a lift from $G/H \sqcup G/K$.

    \begin{equation}\label{Eq_lift in Kan 2}
        \begin{tikzcd}[column sep=2.4em]
        G/H \times \partial \Delta^n \ar[ddd,hook] \ar[dr,"(\emph{id}_{\partial \Delta^n} \times \{1\}) "] \ar[drrr,"\alpha", bend left]&&\\
        & G/H \times \partial \Delta^n \times \Delta^1	\ar[d,hook]	\ar[rr," f"] & & \K\ar[d,"p"]   & G/H \times \partial \Delta^n \ar[rr,"f \circ (\emph{id}_{\partial \Delta^n} \times \{0\}) "] \ar[d,hook] & & \K \ar[d,"p"]\\
        & G/H \times \Delta^n \times \Delta^1 \ar[rr,"g"] & & \cL &   G/H \times \Delta^n \ar[rr,"g \circ( \emph{id}_{\Delta^n} \times \{0\})"] &&\cL\\
        G/H \times \Delta^n \ar[ur,"(\emph{id}_{\Delta^n} \times \{1\}) "{swap}] \ar[urrr,"\beta"{swap},bend right] &&
        \end{tikzcd}
    \end{equation}

\end{lemma}

\begin{proof}
Recall the standard Quillen adjunction 
    $$\hom_{G}(G/H \times \K, \cL) \cong \hom(\K, \cL^H).$$
The diagram on the left of \eqref{Eq_lift in Kan 2} corresponds to the diagram on the left below. From \cite[Theorem I.7.10]{GJ99}, if the diagram on the right below has a lifting, then so does the left below one.
     $$\begin{tikzcd}[column sep=2.4em]
        \partial \Delta^n \ar[ddd,hook] \ar[dr,"(\emph{id}_{\partial \Delta^n} \times \{1\}) "] \ar[drrr,"\alpha' ", bend left]&&\\
        & \partial \Delta^n \times \Delta^1	\ar[d,hook]	\ar[rr," f' "] & & \K^H \ar[d,"p^H"]   & \partial \Delta^n \ar[rr,"f' "] \ar[d,hook] & & \K^H \ar[d,"p^H"]\\
        &\Delta^n \times \Delta^1 \ar[rr,"g' "] & & \cL^H &   \Delta^n \ar[rr,"g' "] &&\cL^H \\
        \Delta^n \ar[ur,"(\emph{id}_{\Delta^n} \times \{1\}) "{swap}] \ar[urrr,"\beta' "{swap},bend right] &&
        \end{tikzcd}$$
    The lifting of the right diagram above comes from the correspondence of the following two diagrams
        $$
        \xymatrix{
        G/H \times \partial \Delta^n \ar[rr]^{\hspace{1cm}f \circ (\emph{id}_{\partial \Delta^n} \times \{0\}) } \ar@{^{(}->} [d]_{} & & \K \ar[d]^{p} & & \partial \Delta^n \ar[rr]^{f'} \ar@{^{(}->} [d]_{} & & \K^H \ar[d]^{p^H}\\
        G/H \times \Delta^n \ar[rr]_{\hspace{1cm}g \circ (\emph{id}_{\partial \Delta^n} \times \{0\}) } & & \cL & & \Delta^n \ar[rr]_{g'} & & \cL^H.
        }
        $$
        
\end{proof}

\begin{lemma}\label{lem:lift}
    Let $\delta \geq 0$ be a real number and $p \colon X \to Y$ in $G$-$\mathbb{S}^{\mathbb{R}}$ induce a $(0,\delta)$-interleaving in homotopy groups. If $X$ and $Y$ are projective fibrant and $p$ is a projective fibration ( $p_r^H\colon X_r^H \to Y_r^H$ is a Kan fibration for all $H\leq G$ ), then $p$ has the right $2 \delta$-lifting property with respect to boundary inclusions 
        $$r \odot \bigsqcup \limits_{H \leq G}(G/H \times \partial \Delta^n) \to r \odot \bigsqcup \limits_{H \leq G}(G/H \times \Delta^n)$$ 
    for every $r \in \RR$, every $n \in \NN $ and for every  subgroup $H$ of $G$ .
\end{lemma}

\begin{proof}
    The commutative diagram on the left corresponds to the one on the right below, 
    $$
        \xymatrix{
        r \odot (G/H \times \partial \Delta^n) \ar[rr]^{a} \ar@{^{(}->} [d]_{} & & X \ar[d]^{p} & & G/H \times \partial \Delta^n \ar[rr]^{\alpha} \ar@{^{(}->} [d]_{} & & X_r \ar[d]^{p_r}\\
        r \odot (G/H \times \Delta^n) \ar[rr]_{b} & & Y & & G/H \times \Delta^n \ar[rr]_{\beta} & & Y_r
        }
    $$
    Now, the commutative diagram on the right-hand side above corresponds to the commutative diagram below 
    $$
        \xymatrix{
        \partial \Delta^n \ar[rr]^{\alpha'} \ar@{^{(}->} [d]_{} & & X_r^H \ar[d]^{p_r^H}\\
         \Delta^n \ar[rr]_{\beta'} & & Y_r^H . 
        }
    $$
    Since $X_r^H,Y_r^H$ are topological spaces and $p_r^H$ is a Kan fibration; using \cite[Lemma 13]{Jar20}, we get a $2\delta$ lift $\Delta^n\to X_{r+2\delta}^H$ for $p^H$  for all subgroups $H$ of $G$ as follows
     $$
        \xymatrix{
        \partial \Delta^n \ar[rr]^{\alpha'} \ar@{^{(}->} [d]_{} & & X_r^H \ar[d]\ar[rr]^{} \ar[d] & & X_{r+2\delta}^H \ar[d]^{p_{r+2\delta}^H}\\
         \Delta^n \ar[rr]_{\beta'} \ar@/^.5pc/@{-->}[urrrr] & & Y_r^H \ar[rr]_{} & & Y_{r+2\delta}^H.
        }
    $$
    So we get a lift from $G/H \times \Delta^n \to X_{r+2\delta}$ For all subgroup $H$ of $G$. Since $2\delta$ is a lift for $p_r^H$  for all $H \leq G$ and all orbits are disjoint, we get a lift from $\bigsqcup \limits_{H \leq G} G/H \times\Delta^n \to X_{r+2\delta}$ of $p$.
    
\end{proof}

\begin{corollary}\label{cor:lift}
    Let $\delta \geq 0$ and let $p \colon X \to_\delta Y$ induce a $\delta$-interleaving in homotopy groups. If $X$ and $Y$ are projective fibrant and $p$ is a projective fibration, then $p$ has the right $(4(n+1) \delta)$-lifting property with respect to equivariant $n$-cofibrations for all $n \in \NN$.
\end{corollary}

\begin{proof}
    From our assumption and Remark \ref{Rem_delta to 2delta}, $p\colon X \to Y(\delta)$ induces $(0,2\delta)$-interleaving in all homotopy groups. An equivariant $n$-cofibration can be written as a composite of $n+1$ equivariant single extensions. 
    A single-dimensional extension is the pushout of a coproduct of
    $$
        \bigsqcup\limits_{i \in I} r_i \odot (\bigsqcup\limits_{j \in J}G/H_j \times \partial D^n) \to \bigsqcup\limits_{i \in I} r_i \odot (\bigsqcup\limits_{j \in J}G/H_j \times D^n)
    $$
    Now, using the universal property of coproducts and Lemma~\ref{lem:lift}, we can say that $f$ has right $4\delta$-lifting property with respect to coproducts of the above form. Hence, $f$ has right $4(n+1)\delta$-lifting property.
    
\end{proof}

\begin{theorem}\label{Th_interleaving in homotopy category}
    Let $X, Y \in G$-$\mathbb{S}^{\RR}$ be persistent spaces that are assumed to be projective cofibrant and pointwise $n$-skeletal if $G$-$\mathbb{S}={\GsS}$, or persistent $G$-$CW$-complexes of dimension at most $n$ if $G$-$\mathbb{S}={\GTop}$. 
    Let $\delta \geq 0$ be a real number. If there exists a map $f \colon X \to_{\delta} Y$ in $G$-$\mathbb{S}^{\RR}$ that induces $\delta$-interleavings in all homotopy groups, then $X$ and $Y$ are $(4(n+1)\delta)$-interleaved in the homotopy category ${\mbox{Ho}}(G$-$\mathbb{S}^{\RR})$. 
\end{theorem}

\begin{proof}
    From Proposition~\ref{Prop_G-Sset-n-cofibrant} and Proposition~\ref{Prop_G-CW-n-cofibrant}, we have that $X$ and $Y$ are equivariant $n$-cofibrant. Given that $f$ induces $\delta$-interleavings in all homotopy groups. We start with a convenient representative $p$ from $[f]\in \mbox{Ho}(G$-$\mathbb{S}^\RR)$ as follows.       
    From the model category axioms, we can get $p\colon X'\to Y'(\delta)$ a projective fibration between projective fibrant objects such that there exist trivial cofibrations $i\colon X \to X'$ and $j\colon Y \to Y'$ with $[p] \circ [i] = [j_\delta] \circ [f]$ in $\mbox{Ho}(G$-$\mathbb{S}^\RR)$, i.e.,
        $$
        \xymatrix{
        X \ar[rr]^{f} \ar[d]_{i} \ar@{-->}[drr]^{j_\delta \circ f} & & Y(\delta)\ar[d]^{j_\delta}\\
        X' \ar[rr]_{p} & & Y'(\delta).  
        }
        $$
    Since $p \in [f]$, it induces $(0,2\delta)$ interleaving in homotopy groups of $X'$ and $Y'(\delta)$. Also note that $Y(\delta)$ is equivariant $n$-cofibrant. Then by using Corollary \ref{cor:lift}, we can find a $4(n+1)\delta$-lift $g'$ of $p$ with respect to the $n$-cofibration $\emptyset \to Y$ as follows
        $$\xymatrix{
        \emptyset \ar[rr]^{} \ar[d]_{} & & X' \ar[d]^{p} \ar[rr] && X'(4(n+1)\delta)\\
        Y(\delta) \ar[rr]_{j_\delta} \ar@/^.6pc/@{-->}[urrrr]_{g'\hspace{4.5cm}} & & Y'(\delta).  
        }$$
    Now, since $j \colon Y \to Y'$ is trivial cofibration and $X'$ is fibrant, by the lifting property of a model category, we can get a lift $g\colon Y' \to X'({(4n+3)\delta})$ as follows
        $$
            \xymatrix{
            Y \ar[rr]^{g'} \ar@{^{(}->}[d]_{j}^{\simeq}  & & X'({(4n+3)\delta}) \ar@{->>}[d]^{}\\
            Y' \ar[rr] \ar@{-->}[rru]^{g}& & \mathbf{*} 
            }
        $$ 
    where $\mathbf{*}$ is the terminal object.
     Now, our claim is $$S_{\delta,4(n+1)\delta}(p)\colon X' \to_{4(n+1)\delta} Y' \quad  \text{and} \quad S_{(4n+3)\delta, 4(n+1)\delta}(g) \colon Y' \to _{4(n+1)\delta} X'$$ form a $4(n+1)\delta$-interleaving in the homotopy category between $X'$ and $Y'$, where $S_{(\epsilon,\delta)}$ denotes the structural maps $A_{\bullet} \to A_{\bullet+(\delta-\epsilon)}$ of the respective diagrams. For simplicity of notation we will also use $S_{\delta} = S_{(0,\delta)}$

    Observe from the construction that $p_{(4n+3)\delta} \circ g \circ j= p_{(4n+3)\delta} \circ g' = S_{4(n+1)\delta} (j)$, where the subscript of $p$ indicates the index of the domain and $S_{4(n+1)\delta}$ is a map $Y' \to Y'(4(n+1)\delta)$ getting by compsotion of structural maps. 
    Since $j$ become an isomorphism in homotopy category, it follows that $[p]_{(4n+3)\delta} \circ [g] = S_{4(n+1)\delta}([id_{Y'}])$. Hence,
        $$
        (S_{\delta,4(n+1)\delta}([p]))_{4(n+1)\delta} \circ S_{(4n+3)\delta, 4(n+1)\delta}([g])= S_{8(n+1)\delta}([id_{Y'}]).
        $$
    Now shifting the relation $[p]_{(4n+3)\delta} \circ [g] = S_{4(n+1)\delta}([id_{Y'}])$  by $\delta$ we get,
        \begin{myeq}\label{imm use}
            [p]_{(4n+1)\delta} \circ [g]_\delta = 
            ([p]_{(4n+3)\delta} \circ [g])_\delta = (S_{4(n+1)\delta}([id_{Y'}]))_\delta 
        \end{myeq}
    where $(S_{4(n+1)\delta}([id_{Y'}]))_\delta $ is the structural map $Y'(\delta) \to Y'({(4n+5)\delta}).$ Now, we compose \eqref{imm use} with $[p]\circ[i]\colon X\to Y'(\delta)$ and by the naturality of the structural map, we have,
        $$
        [p]_{4(n+1)\delta} \circ [g]_\delta \circ [p] \circ [i] = (S_{4(n+1)\delta}([id_{Y'}]))_\delta \circ [p] \circ[i] = [p]_{4(n+1)\delta} \circ S_{0,4(n+1)\delta}([i]).
        $$
    Hence $p_{(4n+1)\delta} \circ g_\delta \circ p \circ i \colon X \to Y'({(4n+5)\delta})$ is homotopic to $p_{4(n+1)\delta} \circ S_{0,4(n+1)\delta}(i)$.
    Let $H\colon I \times X \to Y'({(4n+5)\delta})$ be a homotopy between these maps, which gives the commutative diagram,
        $$
            \xymatrix@C=1.8cm{
            X \bigsqcup X \ar[rr]^{(S_{0,4(n+1)\delta}(i), g_\delta \circ p \circ i)} \ar[d]_{(\iota_0,\iota_1)}  & & X'({4(n+1)\delta}) \ar[d]^{p_{4(n+1)\delta}} \ar[rr] & & X'({8(n+1)\delta})\\
            I \times X \ar[rr]_{H} \ar@/^.3pc/@{-->}[rrrru]& & Y'({(4n+5)\delta})  
            }
        $$
    where $\iota_0, \iota_1$ are level-wise inclusions. We claim that the left vertical map, which is the inclusion into the cylinder, is an equivariant $n$-cofibration. Since $X$ is an equivariant $n$-cofibrant and a cell decomposition of the map is obtained by attaching an equivariant $d+1$-cell corresponding to each equivariant $d$-cell in the decomposition of $X$, the inclusion is an equivariant $n$-cofibration. Using Corollary~\ref{cor:lift}, we can find a $4(n+1)\delta$- lift of the diagram, which shows that,
    $$
    S_{4(n+1)\delta,8(n+1)\delta}( [g]_\delta \circ [p] \circ [i]) =S_{0,8(n+1)\delta}([i])\colon X \to X'({8(n+1)\delta})
    $$
    Now the left-hand side equals $ (S_{(4n+3)\delta,4(n+1)\delta}[g])_{4(n+1)\delta} \circ  S_{\delta,4(n+1)\delta}([p] \circ [i])$ and $[i]$ is an isomorphism, it follows that $(S_{(4n+3)\delta,4(n+1)\delta}[g])_{4(n+1)} \circ  S_{\delta,4(n+1)\delta}([p]) = S_{8(n+1)\delta}[id_{X'}]$.
    
\end{proof}

\begin{remark}
 The proof of the above Theorem \ref{Th_interleaving in homotopy category} also works for $\mathbb{R}^n$.
\end{remark}


\section{Equivariant Persistent Nerve Lemma}\label{Sec_EPNL}

In this section, we prove the equivariant version of the persistent nerve lemma.  As a consequence, we prove an equivariant weak law of large numbers for filtrations.

Nerve lemma is a classical result in homotopy theory. The idea of a nerve has been dated back to Alexandroff\cite{Alex28}. If $\mathcal{U}=\{U_i\}_{i \in \Lambda}$ is an open cover of topological space $T$, then the nerve of $\mathcal{U}$, denoted by $\nrv(\mathcal{U})$, is the simplicial complex whose simplices are the finite subsets $J \subset \Lambda$ such that the intersection $\bigcap_{i \in J} U_i$ is non-empty. The geometric realization of $\nrv(\mathcal{U})$ is denoted by $|\nrv(\mathcal{U})|$.
The nerve lemma says a paracompact space $T$ is homotopy equivalent to the $| \nrv(\mathcal{U})|$ where $\mathcal{U}$ is an open cover of $T$ such that every non-empty finite intersection of sets from $\mathcal{U}$ is contractible, see \cite[Corollary 4G.3]{Hat}. 
Such a cover $\mathcal{U}$ of $T$ is called a \emph{good cover}. The technique follows a well-known scheme in homotopy theory: by introducing an intermediate space and showing it is homotopy equivalence to both $T$ and $\nrv (\mathcal{U})$. 

After the first encounter with the persistent nerve lemma in Chazal--Oudot\cite{CO08}, many articles explored the notion. Finally, in the exposition \cite{BKRR23}, Bauer--Kerber--Roll--Rolle prove it along with several other versions useful for persistent setup. Using a good cover for diagrams (cf. \cite[Definition 6.7]{BL23}), the proof follows a ``unified" scheme of introducing an intermediate space as follows
    $$
        \xymatrix{
        X_t \ar@{^{(}->} [d]_{} & & Z_t \ar[ll]_{\simeq} \ar@{^{(}->}[d] \ar[rr]^{\simeq} \ar[d] & & | \nrv(\mathcal{U}_t) | \ar@{^{(}->}[d]\\
         X_{t'}   & & Z_{t'} \ar[ll]_{\simeq} \ar[rr]^{\simeq} & & | \nrv(\mathcal{U}_{t'}) |,
        }
    $$
where $t \leq t'$ in $\I$ for some small category $\I$ and $(X_t, \mathcal{U}_t)$ is a general filtration with $\mathcal{U}_t$ being a cover of $X_t$.

In the equivariant setup, one can have different versions; see \cite{HH13,Yang14,GG24,LM26} for reference. Similar to the non-equivariant case, we need to consider an ``equivariant good cover". We recall the definition from Lim--Memoli\cite{LM26} since they dealt with a more general topological version of nerve lemma.

\begin{definition}[\cite{LM26}]
    Let $G$ be a finite group and $T$ a paracompact $G$-space with a $G$-invariant open cover $\mathcal{U} \coloneq\{U_i\}_{i \in \Lambda}$ (i.e. $G$ acts on $\Lambda$).
    This induces a simplicial $G$-action on $\nrv(\mathcal{U})$ which maps $U_{\sigma} \coloneq \bigcap_{\ell=0}^k U_{i_{\ell}}$ to $U_{g\sigma} \coloneq \bigcap_{\ell=0}^k U_{gi_{\ell}}$. We denote $\nrv(\mathcal{U})$ with $G$-action by $\Nrv(\mathcal{U})$.
    Let $G_{\sigma}$ be the stabilizer of $\sigma$. 
     If $U_{\sigma}$ is $G_{\sigma}$-contractible for every $\sigma \in \Nrv(\mathcal{U})$, then $\mathcal{U}$ is called an \emph{equivariant good cover} of $T$. (Recall that $U_{\sigma}$ is called $G_{\sigma}$-\emph{contractible} if it is $G_{\sigma}$-homotopy equivalent to a point.)
\end{definition}

In \cite[Proposition 6.3]{LM26}, the authors proved that for an equivariant good cover $\mathcal{U}$ of $T$, there is a $G$-homotopy equivalence $|\Nrv(\mathcal{U})| \simeq_G T$ following a similar technique by introducing an intermediate space, namely \emph{complex of spaces} following the notion used by Hatcher\cite{Hat}. 
Notice that the realization of the complex of spaces is similar to the notion \emph{bar construction} used by Bauer--Kerber--Roll--Rolle\cite{BKRR23} to study a unified version of nerve lemma in the non-equivariant setup. Here we briefly recall the construction and the proof from \cite[Appendix B]{LM26} and use the naturality to prove our desired equivariant persistent nerve lemma.

Let $T$ be a $G$-topological space and $\mathcal{U}$ be a $G$-invariant open cover. For any $\sigma, \tau \in \Nrv(\U)$ where $\sigma$ is a face of $\tau$ (denoted by $\sigma \prec \tau$), there is a canonical inclusion $\iota_{\tau, \sigma} \colon U_{\tau} \to U_{\sigma}$. Hence, $\{U_{\sigma}\}_{\sigma \in \Nrv(\U)}$ and $\{\iota_{\tau, \sigma}\}_{\sigma \prec \tau}$ form a complex of spaces $\Gamma_{\U}$ indexed by the $1$-skeleton of $\sd(\Nrv(\U))$, the barycentric subdivision of $\Nrv(\U)$. Recall that for a directed graph $\Gamma=(V(\Gamma), E(\Gamma))$ with commutative edges, a complex of spaces associated to $\Gamma$ consists of a set of topological spaces $\{T_v\}_{v \in V(\Gamma)}$ and set of embeddings $\{\iota_{u,v} \colon T_u \to T_v\}_{e=(u,v)\in E(\Gamma)}$.

The bar construction of $\Gamma_{\U}$ is denoted by 
    \begin{myeq}\label{Eq_bar gamma u}
        \br{\Gamma_{\U}} \coloneq \Big( \bigsqcup_{\sigma_0 \prec \sigma_1 \prec \cdots \prec \sigma_n} (U_{\sigma_n} \times \Delta^n)  \Big)/ \sim
    \end{myeq}
where $U_{\sigma_n} \times \Delta^n \ni (x, \iota_{n',n}(p)) \sim (\iota_{\sigma_n, \sigma_{i_{n'}}},p) \in U_{\sigma_{\iota_{n'}}} \times \Delta^{n'}$ whenever $\{\sigma_{i_0}, \sigma_{i_1}, \dots, \sigma_{i_{n'}}\}$ is an $n'$-face of $\{\sigma_0, \sigma_1, \dots, \sigma_n\}$ in $\sd(\Nrv(\U))$. Since the geometric realization of $\sd(\Nrv(\U))$ can be expressed as
    \begin{myeq}\label{Eq_sd nrv u}
        |\sd(\Nrv(\U))|=\Big( \bigsqcup_{\sigma_0 \prec \sigma_1 \prec \cdots \prec \sigma_n} (\{*\} \times \Delta^n)  \Big)/ \sim,
    \end{myeq}
one can view $\br{\Gamma_{\U}}$ as a thickening of $|\sd(\Nrv(\U))|$. 

We also recall the construction of a mapping cylinder complex. For each $\sigma \in \Nrv(\U)$, the intersection $U_{\sigma}$ is $G_{\sigma}$ contractible. We choose a $G_{\sigma}$-homotopy equivalence $\phi_{\sigma} \colon U_{\sigma} \to \{*\}$ and let 
    $$M \phi_{\sigma} \coloneq \big( U_{\sigma} \times [0,1] \big) \sqcup \{*\}/ \sim \quad \text{defined by } (x,1) \sim \phi_{\sigma}(x)$$
be its mapping cylinder. For $\sigma \prec \tau$), the canonical inclusion $\iota_{\tau, \sigma}$ induces a $G_{\sigma}$-map $M \phi_{\tau} \to M \phi_{\sigma}$ compatible with the mapping cylinder structures. These maps assemble into a complex of spaces $M \Gamma_{\U}$ over the $1$-skeleton of $\sd(\Nrv(\U))$, which we call the mapping cylinder complex of $\Gamma_{\U}$. Its bar construction is denoted by
    \begin{myeq}\label{Eq_bar M U}
        \br M \Gamma_{\U} \coloneq \Big( \bigsqcup_{\sigma_0 \prec \sigma_1 \prec \cdots \prec \sigma_n} (M \phi_{\sigma_n} \times \Delta^n) \Big)/ \sim
    \end{myeq}
with the same face identification as in the definition of $\br{\Gamma_{\U}}$ in \eqref{Eq_bar gamma u}.

\begin{remark}
    Since $T$ is a $G$-space, $\U$ is $G$-invariant open cover, then $\br{\Gamma_{\U}}$, $|\sd(\Nrv(\U))|$ and $\br M \Gamma_{\U}$ are $G$-spaces by construction.
\end{remark}

Now we define equivariant good cover of a diagram of $G$-equivariant spaces to prove the equivariant persistent nerve lemma.

    \begin{definition}[Equivariant good cover of a diagram of $G$-equivariant spaces]\label{Defn_equiv good cover of diagram}
        Let $\I$ be a small category and $G$ a finite group. An \emph{equivariant cover of a diagram} $X\colon \I \to \GTop$ of equivariant topological spaces indexed by a set $\mathbf{S}$ with $G$-action is a collection of functors 
            $$\mathcal{U}=\{U^k\colon \I \to \Top\}_{k\in \mathbf{S}}$$ 
        such that for each $t\in \I$, the set $\mathcal{U}_t:=\{U^k_t\mid k \in \mathbf{S}\}$ is an open cover of $X_t$ and invariant under the action of $G$. We call an equivariant cover $\mathcal{U}$ is \emph{good} if following conditions are satisfied:
            \begin{itemize}
                \item each $X_t$ is a paracompact Hausdorff space, 
                 \item $\bigcap\limits_{i\in \sigma}{U^{k_i}_t}$ is $G_\sigma$- contractible for every simplex $\sigma$ of $\Nrv(\{U^k_t\mid k \in \mathbf{S}\})$ for each $t \in \I$.
            \end{itemize}    
    \end{definition}

\begin{remark}\label{Rem_equiv good to good}
    Let $\mathcal{U}$ be an equivariant good cover of the diagram $X \colon \I \to \GTop$. If we forget the group action, then an equivariant good cover becomes a good cover in the sense of Definition 6.7 of \cite{BL23}. Notice that in that article, the authors used closed sets. But one can relax that condition, see \cite[Theorem 5.9]{BKRR23}.
\end{remark}

\begin{theorem}[Equivariant Persistent Nerve Lemma]\label{Th_EPNL}
    If $\mathcal{U}$ is an equivariant good cover of a diagram $X \in \GTI$, then $X$ and $|\Nrv(\mathcal{U})|$ are weakly in $\GTI$.
\end{theorem}

The main idea of the proof is to construct natural maps
    \begin{myeq}
        X_t \xleftarrow{~\rho_{S_t}~} \br \Gamma_{\U_t} \xrightarrow{~\rho_{N_t}~} |\Nrv(\mathcal{U}_t)|,
    \end{myeq}
    
    \noindent for $t \in \I$ where $\rho_{S_t}$ is $G$-weak equivalence and $\rho_{N_t}$ a $G$-homotopy equivalence of $G$-spaces. The map $\rho_{S_t}$ is the natural projection and the other map $\rho_{N_t}$ is constructed in the following Lemma \ref{Lem_bar nerve}, which is essentially Proposition B.1 of \cite{LM26}. Here we briefly recall the proof for the sake of completeness.

\begin{lemma}\cite[Proposition B.1]{LM26}\label{Lem_bar nerve}
    Let $T$ be a paracompact $G$-space and $\U$ an equivariant good cover of $T$. Then there is a $G$-homotopy equivalence
        \begin{myeq}
                \br \Gamma_{\U}~ \simeq_G ~|\Nrv(\mathcal U)|.
        \end{myeq}
\end{lemma}

\begin{proof}
    Recall $\br \Gamma_{\U}$, $|\sd(\Nrv(\U))|$ and $\br M \Gamma_{\U}$ from \eqref{Eq_bar gamma u}, \eqref{Eq_sd nrv u} and \eqref{Eq_bar M U} respectively.
    Note that $|\Nrv(\mathcal U)|$ and $|\sd(\Nrv(\U))|$ are $G$-homotopy equivalent. We show $|\sd(\Nrv(\U))|$ and $\br \Gamma_{\U}$ are $G$-homotopy equivalent to complete the proof. 

     First, we show that $\br M \Gamma_{\U}$ $G$-deformation retracts onto $|\Nrv(\mathcal U)|$. There is a natural $G$-map $\phi \colon \br \Gamma_{\U} \to |\sd(\Nrv(\U))|$ defined by $(x,p) \mapsto (*,p)$, for $x \in U_{\sigma_n}$ and $p \in \Delta^n$. The mapping cylinder is canonically $G$-homeomorphic to $\br M \Gamma_{\U}$ and consequently, $\br M \Gamma_{\U}$ $G$-deformation retracts onto $|\sd(\Nrv(\U))|\cong_G |\Nrv(\mathcal U)|$. Thus
        \begin{myeq}\label{Eq_bar M anr nrv}
        \br M \Gamma_{\U} \simeq_G |\Nrv(\mathcal U)|.
        \end{myeq}

    Next we show that $\br M \Gamma_{\U}$ $G$-deformation retracts onto $\br \Gamma_{\U}$. For $k \geq 0$, let $\br M \Gamma_{\U}^{(k)}$ denote the part of $\br M \Gamma_{\U}$ lying over the $k$-skeleton of $\sd(\Nrv(\U))$, i.e.,
        $$\br M \Gamma_{\U}^{(k)} \coloneq \Big( \bigsqcup_{\substack{\alpha_0 \prec \alpha_1 \prec \cdots \prec \alpha_n \\ n \leq k}} M \phi_{\sigma_n} \times \Delta^n \Big)/ \sim;$$
    so that $\br M \Gamma_{\U}^{(-1)}=\emptyset$ and $\br M \Gamma_{\U}= \cup_{k \geq 0} \br M \Gamma_{\U}^{(k)}$. Now we claim that for each $k \geq 0$ the inclusion 
        $$\br M \Gamma_{\U}^{(k-1)} \cup \br \Gamma_{\U} \hookrightarrow \br M \Gamma_{\U}^{(k)} \cup \br \Gamma_{\U}$$
    is a $G$-homotopy equivalence. To see this, let us fix a $k$-simplex $\hat{\tau}=\{\tau_0, \tau_1, \dots, \tau_k\}$ of $\sd(\Nrv(\U))$. Then over $\hat{\tau}$, the space $\br M\Gamma_{U}$ restricts to $M \phi_{\tau_k} \times \Delta^k$, while the union of the part over $\br \Gamma_{\U}$ and over the $(k-1)$ faces corresponds to 
        $$(U_{\tau_k} \times \Delta^k) \cup (M \phi_{\tau_k} \times \partial \Delta^k).$$
    By \cite[Lemma B.4]{LM26}, there is a $G_{\tau_k}$-deformation retraction of $M \phi_{\tau_k} \times \Delta^k$ onto $(U_{\tau_k} \times \Delta^k) \cup (M \phi_{\tau_k} \times \partial \Delta^k)$, using the $G_{\tau_k}$-contractibility of $U_{\tau_k}$. These local retractions are compatible on overlaps and by the equivariant homotopy extension property (see \cite[Lemmas B.1 and B.2]{LM26}), assemble into a $G$-deformation retract
        $$\br M \Gamma_{\U}^{(k)} \cup \br \Gamma_{\U} \to \br M \Gamma_{\U}^{(k-1)} \cup \br \Gamma_{\U}.$$
    Thus they are $G$-homotopy equivalence as per our claim. The infinite concatenation of these deformation retracts, by letting k vary with the $k$-th $G$-deformation retraction performed during the time interval $[1/2^{k+1},1/2^k]$, yields a $G$-deformation retraction $\br M \Gamma_{\U} \to \br \Gamma_{\U}$. In particular
        \begin{myeq}\label{Eq_bar M and bar U}
            \br M \Gamma_{\U} \simeq_G \br \Gamma_{\U}.
        \end{myeq}
    Combining \eqref{Eq_bar M anr nrv} and \eqref{Eq_bar M and bar U} we get result.     
    
\end{proof}

\begin{proof}[Proof of Theorem \ref{Th_EPNL}]
    Let $\U_t$ be a cover of $X_t$ for all $t\in \I$ as defined in Definition \ref{Defn_equiv good cover of diagram}. Then $\U_t$ is an equivariant good cover of the $G$-space $X_t$. By Lemma \ref{Lem_bar nerve}, we have a $G$-homotopy equivalence and denote it by $\br \Gamma_{\U_t} \xrightarrow[]{ ~\rho_{N_t}~} |\Nrv(\U_t)|.$
    On the other hand the canonical projection $\br \Gamma_{\U_t} \xrightarrow[]{~ \rho_{S_t} ~} X_t$ defined by $(x,p) \mapsto x$ is a $G$-weak equivalence. Combining these two, we obtain a zig-zag of $G$-weak equivalences
        \begin{myeq}\label{Eq_zig zag}
        X_t \xleftarrow{~\rho_{S_t}~} \br \Gamma_{\U_t} \xrightarrow{~\rho_{N_t}~} |\Nrv(\mathcal{U}_t)|
        \end{myeq}
    for every $t\in \I$. Since the bar construction is functorial, the canonical projection in $\rho_{S_t}$ is defined without choices and the mapping cylinder construction is natural as well, both $\rho_{S_t}$ and $\rho_{N_t}$ are natural $G$-maps. Hence the zig-zag in \eqref{Eq_zig zag} assembles into natural transformations 
        $$X \xleftarrow{\ \rho_S\ } \br \Gamma_{\U} \xrightarrow{\ \rho_N\ } |\Nrv(\mathcal U)|.$$
    Since each component is a $G$-weak equivalence, $X$ and $|\Nrv(\mathcal U)|$ are $G$-weakly equivalent.
        
\end{proof}

\begin{mysubsection}{A weak law of large numbers for equivariant filtrations}
As an application of our Equivariant Persistent Nerve Lemma (Theorem \ref{Th_EPNL}), we prove an equivariant version of weak law of large numbers for equivariant \v{C}ech filtrations. 
The weak law of large numbers says that, as the number of random samples increases, the sample statistics gets closer in probability to the statistics of the overall data. In \cite{BL23}, the authors showed that under suitable hypotheses, filtrations built from finite data converge to the filtration of the underlying space in an appropriate homotopical sense. Since many geometric objects naturally admit finite group symmetries, it is desirable to establish analogous convergence results in the equivariant setting. Such results ensure that the limiting homotopy type and its symmetries are simultaneously recovered from random samples. In this section,  we use equivariant persistent nerve lemma to obtain an equivariant weak law of large numbers for filtrations.

Let $M$ be a compact Riemannian manifold and $G$ be a finite group acting on $M$ by isometries. Then by an equivariant version of Nash's Theorem from \cite{MS80}, we can assert that there is an orthogonal representation of $G$ on some Euclidean space $\Rn$ and an isometric embedding from $M$ into $\Rn$ which is equivariant with respect to the representation. Using the above embedding, we can assume $M \subset \Rn $ is an $m$-dimensional compact Riemannian manifold. Consider $d_{\Rn}$ be the Euclidean distance on $\Rn$. We consider $M$ to be a probability space endowed with the normalised $m$-dimensional Hausdorff measure $\mu$. Since $G$ acts by isometries on $M$, the measure $\mu$ is $G$-invariant. Let $P_r = \{ p_1, p_2, \cdots, p_r\}$ be an i.i.d sample of $r$-points of $M$. In general the sample need not be invariant under the action of $G$. We therefore, take the orbit
    $$\Orb(P_r)=\{g.p_i \mid g\in G ,1\leq i\leq r \}\subseteq M $$
which is a $G$-invariant subset of $M$. Without loss of generality, we consider $M$ to be an infinite set, so that the orbit space cannot become the entire manifold $M$.

For a fixed $p \in [1, \infty]$, $x \in \RR^n$ and $r \in \RR$, the closed $\ell^p$-ball of radius $r$ centered at $x$ is denoted by $B(x,t)=\{y \in \RR^n ~\mid~ ||x-y||_p \leq t\}$. For $A \subset \RR^n$ with isometric $G$-action, the \emph{offset filtration} $O(A) \colon [0,\infty) \to \GTop$ is defined by 
    $$O(A)_t \coloneq \bigcup_{x \in A} B(x,t)$$
where the $G$-action on $O(A)_t$ is induced by the $G$-action on $A$.
Notice that for a compact $A$, $O(A)=\mathscr{S}(d_A)$, where $\mathscr{S}$ is sublevel set filtration and $d_A \colon \RR^n \to \RR$ is the equivariant distance function to $A$ where $\RR^n$ is equipped with induced $G$-action from that of $A$ and $G$ acts trivially on $\RR$. The set of offset filtrations $\mathcal{U} \coloneq \{O(\{x\}) ~|~ x \in A\}$ is an equivariant cover of $O(A)$. We call the nerve $\Nrv(\mathcal{U})$ the \emph{\v{C}ech filtration} of $A$ and denote it by $\check{C}(A)$.
If $A$ is finite, then $\mathcal{U}$ is an equivariant good cover of $O(A)$. Also for finite $A$, by Theorem \ref{Th_EPNL}, we have $\check{C}(A) \simeq_{G} O(A)$.

\begin{proposition}\label{Prop_weak law}
    Let $M$ be a compact Riemannian manifold and $G$ be a finite group acting on $M$ by isometries. Assume $d^G$ is a stable and homotopy invariant distance on $\GTR$. Then, with notations as above, $d^G(\mathcal{\check{C}}(\text{Orb}(P_r)),O(M))$ converges in probability to $0$ as $r\to\infty$.
\end{proposition}

\begin{proof}
    Let $\epsilon>0$. We will show that
        $$\underset{r\to\infty}{\lim} \hspace{0.1cm} \mathbb{P}(d^G(\mathcal{\check{C}}(\Orb(P_r)),O(M))>\epsilon)=0.$$
    Consider, $\mathcal{U}_G=\{U_\alpha\}_{\alpha\in \Lambda}$ a collection of open balls of radius $\epsilon/2$ which is invariant under the action of $G$ and also cover $M$. Since $M$ is compact there exists a finite sub-cover $U_1,U_2 \dots,U_s$ of $M$. This sub-cover need not be $G$-invariant. Let
    $$\mathcal{U'} \coloneq \{g.U_i \mid g\in G,; 1\le i\le s,\}.$$
    Since the action of $G$ is by isometries, each $g.U_i$ is again a ball of the same radius. The collection $\mathcal{U}'_G$ is $G$-invariant and still covers $M$. Moreover, since both $G$ and $\{U_1,\ldots,U_s\}$ are finite and also $\mathcal{U}_G$ is $G$-invariant, $\mathcal{U}'_G$ is a finite $G$-invariant sub-cover of $M$. Hence, for any $G$-invariant cover, there exists a finite $G$-invariant sub-cover. Without loss of generality, we can take $\{U_1,\ldots,U_s\}$ to be a sub-cover invariant under the action of $G$. Consider 
        $$ \ell=\text{min}\{ \mu(U_1),\mu(U_2),\dots,\mu(U_s) \}.$$
    Note that $\ell>0$. Let $E^i_r$ be the event $\Orb(P_r) \cap U_i=\emptyset$ and let
        $$E_r=\bigcup_{i=1}^{s}E^i_r.$$
    In the complement of the event $E_r$, we have $d^G_{\text{H}}(\Orb(P_r),M)\leq \epsilon$, where $d^G_{\text{H}}$ is the equivariant Hausdorff distance. 
    
    For any nonempty subset $P\subset \Rn$, consider $d_P\colon \Rn \to \RR $ defined as, $d_P(x)= d_{\Rn}(x,P)$ is the distance function to $P$. Since the action on $\Rn$ is an isometric action, two maps $d^G_{\Orb(Pr)}(x)=d_{\Rn}(x,\Orb(P_r)) \quad \text{and} \quad d^G_M(x)=d_{\Rn}(x,M)$ are both $G$-maps. Since $d^G_H(\Orb(P_r),M)\leq \epsilon$, we have $\|d^G_{\Orb(Pr)}-d^G_M\|_\infty \leq \epsilon$.
    By assumption $d$ is stable, Hence
        $$d^G(O(\text{Orb}(P_r)),O(M)) \leq \|d^G_{\Orb(Pr)}-d^G_M\|_\infty \leq \epsilon$$
    The equivariant persistent nerve theorem assert that, $\mathcal{\check{C}}(\text{Orb}(P_r))\simeq_G O(\text{Orb}(P_r))$. Using the homotopy invariance of $d^G$ and previous inequality, we have
        $$d^G(\mathcal{\check{C}}(\Orb(P_r)),O(M))=d^G(O(\text{Orb}(P_r)),O(M))\leq \epsilon.$$
    Thus,
        $$\mathbb{P}(d^G(\mathcal{\check{C}}(\Orb(P_r)),O(M))>\epsilon) \leq \mathbb{P}(E_r) \leq \sum_{i=1}^s \mathbb{P}(E^i_r)\leq s(1-\ell)^r.$$
    Taking $r\to \infty$ we get,
        $$\underset{r\to \infty}{\lim} \hspace{0.1cm} \mathbb{P}(d^G(\mathcal{\check{C}}(\Orb(P_r)),O(M))>\epsilon) \leq \underset{r\to \infty}{lim} \hspace{0.1cm} s(1-\ell)^r =0.$$   
\end{proof}
\end{mysubsection}
\subsection*{Acknowledgements}
The first author thanks IIT Kanpur for its PhD fellowship. The second author thanks IIT Kanpur for its postdoctoral fellowship.


\begin{thebibliography}{99}

\bibitem{ALMSS24}
    H. Adams, E. Lagoda, M. Moy, N. Sadovek, and A. D. Saha, \emph{Persistent equivariant cohomology}, arxiv:~2408.17331, 2024.

\bibitem{Alex28}
    P. Alexandroff, \emph{\"Uber den allgemeinen {D}imensionsbegriff und seine {B}eziehungen zur elementaren geometrischen {A}nschauung}, Math. Ann. \textbf{98} (1928), no.~1, 617--635. \MR{1512423}

\bibitem{BKRR23}
    U. Bauer, M. Kerber, F. Roll, and A. Rolle, \emph{A unified view on the functorial nerve theorem and its variations}, Expo. Math. \textbf{41} (2023), no.~4, Paper No. 125503, 52. \MR{4652781}

\bibitem{BFGQ19}
    M. G. Bergomi, P. Frosini, D. Giorgi, and N. Quercioli, \emph{Towards a topological–geometrical theory of group equivariant non-expansive operators for data analysis and machine learning}, Nature Machine Intelligence \textbf{1} (2019), 423–433.

\bibitem{BL23}
    A. J. Blumberg and M. Lesnick, \emph{Universality of the homotopy interleaving distance}, Trans. Amer. Math. Soc. \textbf{376} (2023), no.~12, 8269--8307. \MR{4669297}

\bibitem{BBI01}
    D. Burago, Y. Burago, and S. Ivanov, \emph{A course in metric geometry}, Graduate Studies in Mathematics, vol.~33, American Mathematical Society, Providence, RI, 2001. \MR{1835418}

\bibitem{Carlsson09}
    G. Carlsson, \emph{Topology and data}, Bull. Amer. Math. Soc. (N.S.)  \textbf{46} (2009), no.~2, 255--308. \MR{2476414}

\bibitem{CIdeSZ08}
    G. Carlsson, T. Ishkhanov, V. de~Silva, and A. Zomorodian, \emph{On the local behavior of spaces of natural images}, Int. J. Comput. Vis.  \textbf{76} (2008), no.~1, 1--12. \MR{3715451}

\bibitem{CSZ10}
    G. Carlsson, G. Singh, and A. Zomorodian, \emph{Computing  multidimensional persistence}, J. Comput. Geom. \textbf{1} (2010), no.~1, 72--100. \MR{2770959}

\bibitem{CCGGO09}
    F. Chazal, D. Cohen-Steiner, M. Glisse, L. J. Guibas, Steve Oudot, \emph{Proximity of Persistence Modules and their Diagrams}, [Research Report] RR-6568, INRIA. 2008, pp.29. (inria-00292566v4)


\bibitem{CSGO16}
    F. Chazal, V. de~Silva, M. Glisse, and S. Oudot, \emph{The  structure and stability of persistence modules}, SpringerBriefs in  Mathematics, Springer, [Cham], 2016. \MR{3524869}


\bibitem{CO08}
    F. Chazal and S. Oudot, \emph{Towards persistence-based  reconstruction in {E}uclidean spaces}, Computational geometry ({SCG}'08), ACM, New York, 2008, pp.~232--241. \MR{2504289}

\bibitem{CEH07}
    D. Cohen-Steiner, H. Edelsbrunner, and J. Harer, \emph{Stability of persistence diagrams}, Discrete Comput. Geom. \textbf{37} (2007), no.~1, 103--120. \MR{2279866}

\bibitem{DWJS95}
    W.G. Dwyer and J. Spalinski, \emph{Chapter 2 - homotopy theories and model categories}, Handbook of Algebraic Topology (I.M. JAMES, ed.), North-Holland, Amsterdam, 1995, pp.~73--126.

\bibitem{ELZ02}
    H. Edelsbrunner, D. Letscher, and A. Zomorodian, \emph{Topological persistence and simplification}, vol.~28, 2002, Discrete and computational geometry and graph drawing (Columbia, SC, 2001), pp.~511--533. \MR{1949898}

\bibitem{Forsini15}
    P. Frosini, \emph{{$G$}-invariant persistent homology}, Math. Methods Appl. Sci. \textbf{38} (2015), no.~6, 1190--1199. \MR{3338143}

\bibitem{FLM19}
    P. Frosini, C. Landi, and F. M\'emoli, \emph{The persistent homotopy type distance}, Homology Homotopy Appl. \textbf{21} (2019), no. 2, 231--259, MR 3923782

\bibitem{GJ99}
    P. G. Goerss and J. F. Jardine, \emph{Simplicial homotopy theory}, Progress in Mathematics, vol. 174, Birkh\"auser Verlag, Basel, 1999. \MR{1711612}

\bibitem{GG24}
    E. J. Gonz\'alez and J. Gonz\'alez, \emph{Equivariant nerve lemma, simplicial difference, and models for configuration spaces on simplicial complexes}, Topology Appl. \textbf{341} (2024), Paper No. 108749, 16. \MR{4655735}

\bibitem{Hat}
    A. Hatcher, \emph{Algebraic Topology}, Cambridge University Press,  Cambridge, 2002. \MR{1867354}

\bibitem{HH13}
    D. Hess and B. Hirsch, \emph{On the topology of weakly and strongly separated set complexes}, Topology Appl. \textbf{160} (2013), no.~2, 328--336. \MR{3003329}

\bibitem{Hir03}
    P. S. Hirschhorn, \emph{Model categories and their localizations},  Mathematical Surveys and Monographs, vol.~99, American Mathematical Society, Providence, RI, 2003. \MR{1944041}

\bibitem{Hovey}
    M. Hovey, \emph{Model categories}, Mathematical Surveys and Monographs, vol.~63, American Mathematical Society, Providence, RI, 1999. \MR{1650134}

\bibitem{Jar20}
    J. F. Jardine, \emph{Persistent homotopy theory}, 2020, Preprint.

\bibitem{Kozlov08}
    D. Kozlov, \emph{Combinatorial algebraic topology}, Algorithms and  Computation in Mathematics, vol.~21, Springer, Berlin, 2008. \MR{2361455}

\bibitem{LS23}
    E. Lanari and L. Scoccola, \emph{Rectification of interleavings and a persistent {W}hitehead theorem}, Algebr. Geom. Topol. \textbf{23} (2023), no.~2, 803--832. \MR{4587317}

\bibitem{Lesnick15}
    M. Lesnick, \emph{The theory of the interleaving distance on  multidimensional persistence modules}, Found. Comput. Math. \textbf{15}(2015), no.~3, 613--650. \MR{3348168}

\bibitem{LMSM86}
    L. G. Lewis, Jr., J. P. May, M. Steinberger, and J. E. McClure, \emph{Equivariant stable homotopy theory}, Lecture Notes in Mathematics, vol. 1213, Springer-Verlag, Berlin, 1986, With contributions by J. E. McClure. \MR{866482}

\bibitem{LM26}
    S. Lim and F. M\'emoli, \emph{The $G$-Gromov-Hausdorff distance and equivariant topology}, arxiv:~2506.15414, 2026.

\bibitem{Matumoto71}
    T. Matumoto, \emph{On {$G$}-{${\rm CW}$} complexes and a theorem of {J}. {H}. {C}. {W}hitehead}, J. Fac. Sci. Univ. Tokyo Sect. IA Math. \textbf{18}(1971), 363--374. \MR{345103}

\bibitem{May96}
    J. P. May, \emph{Equivariant homotopy and cohomology theory}, CBMS Regional Conference Series in Mathematics, vol.~91, Conference Board of the Mathematical Sciences, Washington, DC; by the American Mathematical Society, Providence, RI, 1996, With contributions by M. Cole, G. Comeza\~na, S. Costenoble, A. D. Elmendorf, J. P. C. Greenlees, L. G. Lewis, Jr., R. J. Piacenza, G. Triantafillou, and S. Waner. \MR{1413302}

\bibitem{MS80}
    J. D. Moore and R. Schlafly, \emph{On equivariant isometric  embeddings}, Math. Z. \textbf{173} (1980), no.~2, 119--133. \MR{583381}

\bibitem{MS10}
    G. Mukherjee and D. Sen, \emph{Equivariant simplicial cohomology with local coefficients and its classification}, Topology Appl. \textbf{157}(2010), no.~6, 1015--1032. \MR{2593715}

\bibitem{MKKWP19}
    G. Muszynski, K. Kashinath, V. Kurlin, M. Wehner, and Prabhat,  \emph{Topological data analysis and machine learning for recognizing  atmospheric river patterns in large climate datasets}, Geoscientific Model Development \textbf{12} (2019), no.~2, 613--628.

\bibitem{Quillen67}
    D. Quillen, \emph{Homotopical algebra}, Lecture Notes in Mathematics, vol. 43, Springer-Verlag, Berlin-New York, 1967. \MR{223432}

\bibitem{Quillen72}
    D. Quillen, \emph{Higher algebraic {$K$}-theory. {I}}, Algebraic  {$K$}-theory, {I}: {H}igher {$K$}-theories ({P}roc. {C}onf., {B}attelle {M}emorial {I}nst., {S}eattle, {W}ash., 1972), Lecture Notes in Math., Vol. 341, Springer, Berlin-New York, 1973, pp.~85--147. \MR{338129}

\bibitem{ER14}
    E. Riehl, \emph{Categorical homotopy theory}, New Mathematical Monographs, vol.~24, Cambridge University Press, Cambridge, 2014. \MR{3221774}

\bibitem{SN09}
    N. P. Strickland, \emph{The category of CGWH spaces}, 2009, Preprint.

\bibitem{Yang14}
    H. Yang, \emph{Equivariant cohomology and sheaves}, J. Algebra \textbf{412} (2014), 230--254. \MR{3215956}

\bibitem{ZC08}
    A. Zomorodian and G. Carlsson, \emph{Localized homology}, Comput. Geom., \textbf{41}(2008), no. 3, 126--148, \MR{2442490}


\end{thebibliography}
\end{document}